\documentclass{SCAE}%SCAEOL for online version; SCAE for publication version; SCAES for the paper dedicated to somebody.
\numberwithin{equation}{section}
\allowdisplaybreaks

\def\rr{{\mathbb R}}
\def\rn{{{\rr}^n}}

\def\zz{{\mathbb Z}}
\def\nn{{\mathbb N}}
\def\cc{{\mathbb C}}

\def\cx{{\mathcal X}}
\def\cm{{\mathcal M}}
\def\cb{{\mathcal B}}

\def\fz{\infty}
\def\az{\alpha}
\def\lz{\lambda}
\def\dz{\delta}
\def\ez{\epsilon}
\def\kz{\kappa}
\def\bz{\beta}
\def\fai{\varphi}
\def\gz{{\gamma}}
\def\oz{{\omega}}

\def\wz{\widetilde}

\def\hs{\hspace{0.3cm}}

\def\ls{\lesssim}

\def\rbmo{{\mathop\mathrm{RBMO}}}
\def\rblo{{\mathop\mathrm{RBLO}}}

\def\supp{{\mathop\mathrm{\,supp\,}}}
\def\loc{{\mathop\mathrm{\,loc\,}}}

\def\einf{{\mathop{\mathrm{\,ess\,inf\,}}}}
\def\diam{{\mathop\mathrm{\,diam\,}}}

\def\dint{\displaystyle\int}

\def\dfrac{\displaystyle\frac}

\def\r{\right}
\def\lf{\left}
\def\gfz{\genfrac{}{}{0pt}{}}

\begin{document}

%Basic Information
\Year{2012} %
\Month{January}
\Vol{55} %
\No{1} %
\BeginPage{1} %
\EndPage{XX} %
\AuthorMark{Lin H {\it et al.}}
\ReceivedDay{July 27, 2012}
\AcceptedDay{April 03, 2013}
\DOI{10.1007/s11425-000-0000-0} % The author doesn't need fill in it.

% \title[short text for running head]{full title}{comments for title}
\title{Equivalent boundedness of Marcinkiewicz integrals
on non-homogeneous metric measure spaces}{}

% \author[]{Full name}{footnote}
% Remark:  One \author for one author

\author{LIN Haibo$^1$}{}
\author{YANG Dachun$^{2,}$}{Corresponding author}
%\author[4]{FIRST3 Last Name3}{}

%
\address[{\rm1}]{College of Science, China Agricultural University, Beijing {\rm 100083}, People's Republic of China;}
\address[{\rm2}]{School of Mathematical Sciences, Beijing Normal University,
Laboratory of Mathematics\\
and Complex Systems, Ministry of
Education, Beijing {\rm 100875}, People's Republic of China;}
%\address[{\rm3}]{Department of Mathematics, University3, City3 {\rm100003}, Country3;}
%\address[{\rm4}]{College of Science, University4, City4 {\rm100004}, Country4}
\Emails{haibolincau@126.com, dcyang@bnu.edu.cn}\maketitle

%     Abstract is required.

 {\begin{center}
\parbox{14.5cm}{\begin{abstract}
 %Abstract is required.
 Let $({\mathcal X},\,d,\,\mu)$ be a metric measure
space satisfying the upper doubling condition and the geometrically
doubling condition in the sense of T. Hyt\"onen. In this paper, the
authors prove that the $L^p(\mu)$ boundedness with $p\in(1,\,\infty)$
of the Marcinkiewicz integral is equivalent to either of its
boundedness from $L^1(\mu)$ into $L^{1,\infty}(\mu)$ or from the atomic
Hardy space $H^1(\mu)$ into $L^1(\mu)$. Moreover, the authors show that,
if the Marcinkiewicz integral is bounded from
$H^1(\mu)$ into $L^1(\mu)$, then it is also bounded from $L^\infty(\mu)$ into
the space ${\mathop\mathrm{RBLO}}(\mu)$ (the regularized {\rm BLO}),
which is a proper subset of
${\rm RBMO}(\mu)$ (the regularized {\rm BMO}) and, conversely, if
the Marcinkiewicz integral is bounded from $L_b^\infty(\mu)$ (the
set of all $L^\infty(\mu)$ functions with bounded support) into the space ${\rm
RBMO}(\mu)$, then it is also bounded from the finite atomic
Hardy space $H_{\rm fin}^{1,\,\infty}(\mu)$ into $L^1(\mu)$.
These results essentially
improve the known results even for non-doubling measures.
 \vspace{-3mm}
\end{abstract}}\end{center}}

%  Keyword is required.

%  \subjclass is required.
%Please look up your MSC number in the attached file msc2010final2-Oct09.pdf

%%%%%%%%%%%%%%%%%%%%%%%%%%%%%%%%%%%%%%%%%%%%%%%%%%%%%%%%%%%%
\renewcommand{\baselinestretch}{1.2}
\begin{center} \renewcommand{\arraystretch}{1.5}
{\begin{tabular}{lp{0.8\textwidth}} \hline \scriptsize
{\bf Citation:}\!\!\!\!&\scriptsize Lin H, Yang D. Equivalent boundedness of
Marcinkiewicz integrals on non-homogeneous metric measure spaces. Sci China Math, 2012, 55, doi: 10.1007/s11425-000-0000-0\vspace{1mm}
\\
\hline
\end{tabular}}\end{center}

%%%%%%%%%%%%%%%%%%%%%%%%%%%%%%%%%%%%%%%%%%%%%%%%%%%%%%%%%%%%
%% Text of article.
%%%%%%%%%%%%%%%%%%%%%%%%%%%%%%%%%%%%%%%%%%%%%%%%%%%%%%%%%%%%
%    Section headings
\baselineskip 11pt\parindent=10.8pt  \wuhao
\section{Introduction}

It is well known that the Littlewood-Paley
$g$-function is a very important tool in harmonic analysis and the
Marcinkiewicz integral is essentially a Littlewood-Paley
$g$-function. In 1938, as an analogy of the classical Littlewood-Paley
$g$-function without going into the interior of the unit disk, Marcinkiewicz \cite{m38} introduced the
integral on one-dimensional Euclidean space $\rr$, which is today called
the Marcinkiewicz integral, and conjectured that it is bounded
on $L^p([0,\,2\pi])$ for any $p\in(1,\,\fz)$. In 1944, by using a complex
variable method, Zygmund \cite{z44} proved
the Marcinkiewicz conjecture. The
higher-dimensional Marcinkiewicz integral was introduced by Stein
\cite{s58} in 1958.  Let $\Omega$
be homogeneous of degree zero in $\rn$ for $n\geq 2$, integrable and
have mean value zero  on the unit sphere $S^{n-1}$. The
higher-dimensional Marcinkiewicz integral $\cm_{\Omega}$ is then defined by
$$
\cm_{\Omega}(f)(x):=\lf[\dint^{\fz}_0\lf|\int_{|x-y|< t}
\frac{\Omega(x-y)}{|x-y|^{n-1}}f(y)\,dy\r |^2\, \dfrac{dt}{t^3}\r
]^{1/2},\quad x\in\rn.
$$
Stein \cite{s58} proved that, if $\Omega\in {\rm
Lip}_{\alpha}(S^{n-1})$ for some $\alpha\in(0,1]$, then
$\cm_{\Omega}$ is bounded on $L^p(\rn)$ for any $p\in (1,2]$ and
also bounded from $L^1(\rn)$ to $L^{1,\,\infty}(\rn)$. Since then,
many papers focus on the boundedness of this
operator on various function spaces. We refer the reader to see
\cite{tw,dlx,dlz,dxy,hly03,hmy,l08,w05,yz}
for its developments and applications.

On the other hand, many results from real analysis and harmonic analysis on
the classical Euclidean spaces have been extended to the space of
homogeneous type introduced by Coifman and Weiss \cite{cw71}.
Recall that a metric space $(\cx,\,d)$ equipped with a Borel measure
$\mu$ is called a {\it space of homogeneous type}, if
$(\cx,\,d,\,\mu)$ satisfies the following {\it measure doubling
condition} that there exists a positive constant $C_\mu$ such that,
for all balls $B(x,\,r):=\{y\in\cx:\ d(x,\,y)<r\}$ with $x\in\cx$
and $r\in(0,\,\fz)$,
\begin{equation}\label{e1.1}
\mu(B(x,\,2r))\le C_\mu\mu(B(x,\,r)).
\end{equation}
Moreover, it is known that many results concerning the theory of
Calder\'on-Zygmund operators and function spaces remain valid even
for non-doubling measures; see, for example
\cite{t01,t03,ntv,cmy,cs,j05,hly07,yyh}. In
particular, let $\mu$ be a non-negative Radon measure on $\rn$ which
only satisfies the {\it polynomial growth condition} that there
exist positive constants $C_0$ and $\kz\in(0, n]$ such that, for all
$x\in\rn$ and $r\in(0,\,\fz)$,
\begin{equation}\label{e1.2}
\mu(B(x,\,r))\le C_0 r^\kz,
\end{equation}
where $B(x,\,r):=\{y\in\rn:\, |x-y|<r\}$. Such a measure $\mu$ need
not satisfy the doubling condition \eqref{e1.1}. The analysis with
non-doubling measures plays a striking role in solving the
long-standing open Painlev\'e problem by Tolsa in \cite{t03}. In
\cite{hly07}, the authors introduced the Marcinkiewicz integral on
$\rn$ with a measure as in \eqref{e1.2}. Moreover, under the
assumption that the Marcinkiewicz integral is bounded on $L^2(\mu)$,
the authors then obtained its boundedness, respectively, from
$L^1(\mu)$ into $L^{1,\fz}(\mu)$, from the atomic Hardy space
$H^1(\mu)$ into $L^1(\mu)$ or from $L^\fz(\mu)$ into the space
$\rblo(\mu)$.

However, as pointed out by Hyt\"onen in \cite{h10} that the measures
satisfying the polynomial growth condition are different from, not
general than, the doubling measures. Hyt\"onen \cite{h10} introduced
a new class of metric measure spaces which satisfy the so-called
upper doubling condition and the geometrically doubling condition
(see also, respectively, Definitions \ref{d1.1} and \ref{d1.2} below).
This new class of metric measure spaces is called the non-homogeneous
space, which includes both the spaces of
homogeneous type and metric spaces with the measures satisfying
\eqref{e1.2} as special cases.

From now on, we always assume that $(\cx,d,\mu)$ is a
non-homogeneous space in the sense of Hyt\"onen \cite{h10}.
In this setting, Hyt\"onen \cite{h10}
introduced the space $\rbmo(\mu)$, the \emph{space of the regularized {\rm BMO}},
and Hyt\"onen and Martikainen \cite{hm} further established a version of $Tb$ theorem.
Later, Hyt\"onen, Da. Yang and Do. Yang \cite{hyy10} studied the
\emph{atomic Hardy space $H^1(\mu)$} and proved that the dual space of
$H^1(\mu)$ is just the space $\rbmo(\mu)$. Some of results in
\cite{hyy10} were also independently obtained by Bui and Duong
\cite{ad10} via different approaches. Moreover, Lin and Yang
\cite{ly11} introduced the space $\rblo(\mu)$
(the \emph{space of the regularized {\rm BLO}})
and applied this space to the boundedness of the maximal Calder\'on-Zygmund operators.
Several equivalent characterizations for the boundedness of the
Calder\'on-Zygmund operators and the maximal Calder\'on-Zygmund
operators were established in \cite{hlyy, lyy, hmy2,lmy}. Some
weighted norm inequalities for the multilinear Calder\'on-Zygmund
operators were presented by Hu, Meng and Yang in \cite{hmy1}. Very
recently, Fu, Yang and Yuan \cite{fyy} established the boundedness
of multilinear commutators of Calder\'on-Zygmund operators with
$\rbmo(\mu)$ functions on Orlicz spaces. Moreover, by a method
different from the classical one, Lin and Yang \cite{ly12} proved an
interpolation result that a sublinear operator, which is bounded from
$H^1(\mu)$ to $L^{1,\,\fz}(\mu)$ and from $L^\fz(\mu)$ to
$\rbmo(\mu)$, is also bounded on $L^p(\mu)$ for all $p\in(1,\,\fz)$.
More developments on harmonic analysis in this setting can be found
in the monograph \cite{yyh}.

The main purpose of this paper is to generalize and improve the
corresponding results in \cite{hly07} for $\cx:=\rn$ with a measure
$\mu$ as in \eqref{e1.2} to the present setting $(\cx,\,d,\,\mu)$.
Precisely, we prove that the $L^p(\mu)$ boundedness with
$p\in(1,\,\fz)$ of the Marcinkiewicz integral is equivalent to
either of its boundedness from $L^1(\mu)$ into $L^{1,\fz}(\mu)$ or
from the atomic Hardy space $H^1(\mu)$ into $L^1(\mu)$.
As for the endpoint case of $p=\fz$, we show that,
if the Marcinkiewicz integral is bounded from
$H^1(\mu)$ into $L^1(\mu)$, then it is bounded from $L^\fz(\mu)$ into
the space $\rblo(\mu)$, which is a proper subset of
$\rbmo(\mu)$. Moreover, we prove that, if
the Marcinkiewicz integral is bounded from $L_b^\fz(\mu)$ (the
\emph{set of all $L^\fz(\mu)$ functions with bounded support}) into the space ${\rm
RBMO}(\mu)$, then it is also bounded from the finite atomic
Hardy space $H_{\rm fin}^{1,\,\fz}(\mu)$ into $L^1(\mu)$. These
results essentially improve the existing results.

To state our main results, we first recall some necessary notions and
notation. We start with the notion of the upper doubling and
geometrically doubling metric measure space introduced in
\cite{h10}.
\begin{definition}\label{d1.1}
A metric measure space $(\cx,\,d,\,\mu)$ is called {\it upper
doubling}, if $\mu$ is a Borel measure on $\cx$ and there exist a
dominating function $\lz:\ \cx\times(0,\,\fz)\rightarrow (0,\,\fz)$
and a positive constant $C_\lz$ such that, for each $x\in\cx$,
$r\rightarrow \lz(x,\,r)$ is non-decreasing and, for all  $x\in\cx$
and $r\in(0,\,\fz)$,
\begin{equation}\label{e1.3}
\mu(B(x,\,r))\le\lz(x,\,r)\le C_\lz\lz(x,\,r/2).
\end{equation}
\end{definition}

\begin{remark}\label{r1.1}
(i) Obviously, a space of homogeneous type is a special case of
upper doubling spaces, where one can take the dominating function
$\lz(x,\,r):=\mu(B(x,\,r))$. Moreover, let $\mu$ be a non-negative
Radon measure on $\rn$ which only satisfies the polynomial growth
condition. By taking $\lz(x,\,r):= Cr^\kz$, we see that
$(\rn,\,|\cdot|,\,\mu)$ is also an upper doubling measure space.

(ii) It was proved in \cite{hyy10} that there exists a dominating
function $\wz\lz$ related to $\lz$ satisfying the property that
there exists a positive constant $C_{\wz\lz}$ such that
$\wz\lz\le\lz$, $C_{\wz\lz}\le C_{\lz}$ and, for all $x,\,y\in\cx$
with $d(x,\,y)\le r$,
\begin{equation}\label{e1.4}
\wz\lz(x,\,r)\le C_{\wz\lz}\wz\lz(y,\,r).
\end{equation}
Based on this, in this paper, we {\it always assume} that the
dominating function $\lz$ also satisfies \eqref{e1.4}.

\end{remark}

We now recall the notion of the geometrically doubling space
(see, for example, \cite{h10}).

\begin{definition}\label{d1.2}
A metric space $(\cx,\,d)$ is said to be {\it geometrically doubling}, if
there exists some $N_0\in\nn:=\{1,2,\ldots\}$ such that, for
any ball $B(x,\,r)\subset\cx$, there exists a finite ball covering
$\{B(x_i,\,r/2)\}_i$ of $B(x,\,r)$ such that the {\it cardinality}
of this covering is at most $N_0$.
\end{definition}

\begin{remark}\label{r1.2}
Let $(\cx,\,d)$ be a metric space. In \cite[Lemma 2.3]{h10},
Hyt\"onen showed that the following statements are mutually
equivalent:
\begin{enumerate}
\item[(i)] $(\cx,\,d)$ is geometrically doubling.

\item[(ii)] For any $\ez\in(0,\,1)$ and ball $B(x,\,r)\subset\cx$,
there exists a finite ball covering $\{B(x_i,\,\ez r)\}_i$ of
$B(x,\,r)$ such that the cardinality of this covering is at most
$N_0\ez^{-n}$, here and in what follows, $N_0$ is as in Definition
\ref{d1.2} and $n:=\log_2 N_0$.

\item[(iii)] For every $\ez\in(0,\,1)$, any ball $B(x,\,r)\subset\cx$
can contain at most $N_0\ez^{-n}$ centers $\{x_i\}_i$ of disjoint
balls with radius $\ez r$.

\item[(iv)] There exists $M\in\nn$ such that any ball
$B(x,\,r)\subset\cx$ can contain at most $M$ centers $\{x_i\}_i$ of
disjoint balls $\{B(x_i,\,r/4)\}^M_{i=1}$.

\end{enumerate}
\end{remark}

The following coefficients $\dz(B,\,S)$ for all balls $B$ and $S$
were introduced in \cite{h10} as analogues of Tolsa's numbers
$K_{Q,\,R}$ in \cite{t01}; see also \cite{hyy10}.

\begin{definition}\label{d1.3}
For all balls $B\subset S$, let
$$
\dz(B,\,S):=\int_{(2S)\setminus
B}\frac{d\mu(x)}{\lz(c_B,d(x,\,c_B))},
$$
where above and in that follows, for a ball $B:= B(c_B,\,r_B)$ and
$\rho\in(0,\,\fz)$, $\rho B:= B(c_B,\,\rho r_B)$.
\end{definition}

The following atomic Hardy space was introduced in \cite{hyy10} and
a slight different equivalent variant was independently introduced
in \cite{ad10}. In what follows, $L^1_\loc(\mu)$ denotes the {\it
space of all $\mu$-locally integrable functions}.

\begin{definition}\label{d1.4}
Let $\rho\in(1,\,\fz)$ and $p\in (1,\,\fz]$. A function $b\in
L_\loc^1(\mu)$ is called a $(p,\,1)_\lz$-{\it atomic block}, if
\begin{enumerate}
\item[(i)] there exists some ball $B$ such that $\supp (b)\subset B$;

\item[(ii)] $\int_\cx b(x)\,d\mu(x)=0$;

\item[(iii)] for any $j\in\{1,\,2\}$, there exist a function $a_j$
supported on a ball $B_j\subset B$ and $\kz_j\in\cc$ such that
$$
b=\kz_1a_1+\kz_2a_2
$$
and
$$
\|a_j\|_{L^p(\mu)}\le[\mu(\rho B_j)]^{1/p-1}[1+\dz(B_j,\,B)]^{-1}.
$$
Moreover, let
$$
|b|_{H_{\rm atb}^{1,\,p}(\mu)}:=|\kz_1|+|\kz_2|.
$$
\end{enumerate}

\end{definition}

\begin{definition}\label{d1.5}
Let $p\in(1,\,\fz]$.

(1) The {\it space}  $H_{\rm fin}^{1,\,p}(\mu)$ is defined to be the \emph{set of all
finite linear combinations of $(p,\,1)_\lz$-atomic blocks}. The {\it
norm} of $f$ in $H_{\rm fin}^{1,\,p}(\mu)$ is defined by
$$
\|f\|_{H_{\rm fin}^{1,\,p}(\mu)}:= \inf
\lf\{\sum_{j=1}^N|b_j|_{H_{\rm atb}^{1,\,p}(\mu)}:\ f=\sum_{j=1}^N
b_j,\,\mbox{$b_j$ is a $(p,\,1)_\lz$-atomic block, $N\in\nn$}\r\}.
$$

(2) A function $f\in L^1(\mu)$ is said to belong to the {\it atomic
Hardy space} $H_{\rm atb}^{1,\,p}(\mu)$, if there exist
$(p,\,1)_\lz$-atomic blocks $\{b_j\}_{j\in\nn}$ such that
$f=\sum_{j=1}^\fz b_j$ and $\sum_{j=1}^\fz
|b_j|_{H_{\rm atb}^{1,\,p}(\mu)}<\fz$. The {\it norm} of $f$ in
$H_{\rm atb}^{1,\,p}(\mu)$ is defined by
$$
\|f\|_{H_{\rm atb}^{1,\,p}(\mu)}:= \inf
\lf\{\sum_j|b_j|_{H_{\rm atb}^{1,\,p}(\mu)}\r\},
$$
where the infimum is taken over all the possible decompositions of
$f$ as above.

\end{definition}

\begin{remark}\label{r1.3}
(1) It was proved in \cite{hyy10} that, for each $p\in(1,\,\fz]$, the
atomic Hardy space $H_{\rm atb}^{1,\,p}(\mu)$ is independent of the
choice of $\rho$ and that, for all $p\in(1,\,\fz)$, the spaces
$H_{\rm atb}^{1,\,p}(\mu)$ and $H_{\rm atb}^{1,\,\fz}(\mu)$ coincide with
equivalent norms. Thus, in the following, we denote
$H_{\rm atb}^{1,\,p}(\mu)$ simply by $H^1(\mu)$.

(2) When $\mu(\cx)<\fz$, as in the space of homogeneous type, the
constant function having value $[\mu(\cx)]^{-1}$ is also regarded
as a $(p,\,1)_\lz$-atomic block (see \cite[p. 591]{cw77}) and,
moreover, $\|[\mu(\cx)]^{-1}\|_{H^1(\mu)}\le 1$.
\end{remark}

We now recall the definition of the space $\rbmo(\mu)$ introduced in
\cite{h10}.

\begin{definition}\label{d1.6}
Let $\rho\in(1,\,\fz)$. A function $f\in L^1_\loc(\mu)$ is said to
be in the {\it space} $\rbmo(\mu)$, if there exist a positive
constant $C$ and a number $f_B$ for any ball $B$ such that, for all
balls $B$,
\begin{equation*}
\dfrac{1}{\mu(\rho B)}\dint_B |f(y)-f_B|\,d\mu(y)\le C
\end{equation*}
and, for balls $B\subset S$,
\begin{equation*}
|f_B-f_S|\le C[1+\dz(B,\,S)].
\end{equation*}
Moreover, the {\it norm} of $f$ in $\rbmo(\mu)$ is defined to be the
minimal constant $C$ as above and denoted by $\|f\|_{\rbmo(\mu)}$.
\end{definition}

It was proved in \cite[Lemma 4.6]{h10} that the space $\rbmo(\mu)$
is independent of the choice of $\rho$.

Let $K$ be a locally integrable function on
$(\cx\times\cx)\setminus\{(x,\,x):\ x\in\cx\}$. Assume that there
exists a positive constant $C$ such that, for all $x,\,y\in\cx$ with
$x\ne y$,
\begin{equation}\label{e1.5}
|K(x,\,y)|\le C\frac{d(x,\,y)}{\lz(x,d(x,\,y))}
\end{equation}
and, for all $y,\,y'\in\cx$,
\begin{equation}\label{e1.6}
\int_{d(x,\,y)\ge 2d(y,\,y')}\big[|K(x,\,y)-K(x,\,y')|
+|K(y,\,x)-K(y',\,x)|\big]\frac{1}{d(x,\,y)}\,d\mu(x)\le C.
\end{equation}
The Marcinkiewicz integral $\cm(f)$ associated to the above kernel
$K$ is defined by setting, for all $x\in\cx$,
\begin{equation}\label{e1.7}
\cm(f)(x):=\lf[\int_0^\fz\lf|\int_{d(x,\,y)<
t}K(x,\,y)f(y)\,d\mu(y)\r|^2\,\frac{dt}{t^3}\r]^{1/2}.
\end{equation}
Obviously, by taking $\lz(x,\,r):= Cr^n$, we see that, in the
classical Euclidean space $\rn$, if
$$K(x,\,y):=\dfrac{1}{|x-y|^{n-1}}\Omega(x-y)$$
with $\Omega$ being homogeneous of degree zero and $\Omega\in {\rm
Lip}_{\alpha}(S^{n-1})$ for some $\alpha\in(0,1]$, then $K$
satisfies \eqref{e1.5} and \eqref{e1.6}, and $\cm$ in \eqref{e1.7}
is just the Marcinkiewicz integral $\cm_\Omega$ introduced by Stein
in \cite{s58}. Thus, $\cm$ in \eqref{e1.7} is a natural
generalization of the classical Marcinkiewicz integral in the
present setting.

One of main results of this article is as follows.

\begin{theorem}\label{t1.1}
Let $K$ satisfy \eqref{e1.5} and \eqref{e1.6}, and $\cm$ be as in
\eqref{e1.7}. Then the following four statements are equivalent:
\begin{itemize}

\item[{\rm (i)}] $\cm$ is bounded on $L^{p_0}(\mu)$ for some
$p_0\in(1,\,\fz)$;

\item[{\rm (ii)}] $\cm$ is bounded from $L^1(\mu)$ into
$L^{1,\,\fz}(\mu)$;

\item[{\rm (iii)}] $\cm$ is bounded on $L^p(\mu)$ for all
$p\in(1,\,\fz)$;

\item[{\rm (iv)}] $\cm$ is bounded from $H^1(\mu)$ into $L^1(\mu)$.
\end{itemize}
\end{theorem}

Comparing with the corresponding result in \cite{hly07}, Theorem
\ref{t1.1} makes an essential improvement.

As for the endpoint case of $p=\fz$, we obtain the following result.
Recall that $L^\fz_b(\mu)$ denotes the {\it set of all $L^\fz(\mu)$
functions with bounded support} and $\rblo(\mu)$ the regularized {\rm BLO}
space introduced in \cite{ly11} (see also Definition \ref{d2.2} below).

\begin{theorem}\label{t1.2}
Let $K$ satisfy \eqref{e1.5} and \eqref{e1.6}, and $\cm$ be as in
\eqref{e1.7}.
\begin{itemize}

\item[{\rm (i)}] If $\cm$ is bounded from $H^1(\mu)$ into
$L^1(\mu)$, then, for $f\in L^\fz(\mu)$, $\cm(f)$ is either infinite
everywhere or finite $\mu$-almost everywhere, more precisely, if
$\cm(f)$ is finite at some point $x_0\in\cx$, then $\cm(f)$ is
finite $\mu$-almost everywhere and
$$
\|\cm(f)\|_{\rblo(\mu)}\le C\|f\|_{L^\fz(\mu)},
$$
where the positive constant $C$ is independent of $f$.

\item[{\rm (ii)}] If there exists a positive constant $C$ such that, for
all $f\in L^\fz_b(\mu)$,
\begin{equation}\label{e1.8}
\|\cm(f)\|_{\rbmo(\mu)}\le C\|f\|_{L^\fz(\mu)},
\end{equation}
then $\cm$ is bounded from $H_{\rm fin}^{1,\,\fz}(\mu)$ into
$L^1(\mu)$.
\end{itemize}
\end{theorem}

\begin{remark}\label{r1.4}
(i) Recall that it was proved in \cite{ly11}
that $\rblo(\mu)$ is a proper subset
of $\rbmo(\mu)$, which, together with Theorem \ref{t1.2}(i), further implies
that, if $\cm$ is bounded from $H^1(\mu)$ into $L^1(\mu)$,
then, for any $f\in L^\fz(\mu)$, $\cm(f)$ is either infinite
everywhere or
$$\|\cm(f)\|_{\rbmo(\mu)}\ls\|f\|_{L^\fz(\mu)}.$$
This is the known result for the Marcinkiewicz integral
over the classical Euclidean space $\rn$.
Moreover, Lin et al. \cite{lny} constructed a nonnegative function
belonging to $\mathrm{BMO}(\rn)$ but not to $\mathrm{BLO}(\rn)$, which further shows
that our result indeed improves the
known corresponding result even on the classical Euclidean space $\rn$.

(ii) From Theorem \ref{t1.1}, we deduce that, if $\cm$ is bounded from
$H^1(\mu)$ into $L^1(\mu)$, then it is also bounded on
$L^p(\mu)$ for all $p\in(1,\,\fz)$. Since
$L^\fz_b(\mu)\subset L^p(\mu)$ for all $p\in(1,\,\fz)$, we then see
that, if $\cm$ is bounded from $H^1(\mu)$ into
$L^1(\mu)$, then, for any $f\in L^\fz_b(\mu)$, $\cm(f)$ is finite at
some point $x_0\in\cx$, which, together with Theorem \ref{t1.2}(i),
implies that it is bounded from $L^\fz_b(\mu)$ into $\rblo(\mu)$
and hence, by (i) of this remark, it is also bounded from
$L^\fz_b(\mu)$ into $\rbmo(\mu)$.

(iii) In the present setting, it is still unclear
whether the boundedness of sublinear operators on
the atomic Hardy space can be deduced only from their behaviors on
atoms. More precisely, it is unclear whether the uniform
boundedness in some Banach space $\cb$
of a sublinear operator $T$ on all $(\fz,\,1)_\lz$-atoms can guarantee
the boundedness of $T$ from $H^1(\mu)$ to  $\cb$ or not.
Thus, under the assumption of Theorem \ref{t1.2}, it is
unclear whether the Marcinkiewicz integral $\cm$ can extends
boundedly from $H^1(\mu)$ to $L^1(\mu)$ or not.
\end{remark}

This paper is organized as follows. In Section \ref{s2}, under the
assumption that the Marcinkiewicz integral is bounded on
$L^{p_0}(\mu)$ for some $p_0\in(1,\,\fz)$, we then obtain its
boundedness, respectively, from $L^1(\mu)$ to $L^{1,\fz}(\mu)$, from
$H^1(\mu)$ to $L^1(\mu)$, from $L^\fz(\mu)$ to the space
$\rblo(\mu)$ and on $L^p(\mu)$ for all $p\in(1,\,\fz)$; see Theorem \ref{t2.1} below.
From this, we deduce that Theorem \ref{t1.1}(i) implies (ii), (iii) and (iv) of Theorem
\ref{t1.1}, which slightly improves the corresponding result in
\cite{hly07} by relaxing the assumption that the Marcinkiewicz
integral is bounded on $L^2(\mu)$ into that it
is bounded on $L^{p_0}(\mu)$ for some $p_0\in(1,\,\fz)$.

In Section \ref{s3}, we prove Theorem \ref{t1.1}.
Indeed, by Theorem \ref{t2.1} and an obvious fact that Theorem \ref{t1.1}(iii) implies Theorem \ref{t1.1}(i),
to prove Theorem \ref{t1.1}, we only need to
prove that Theorem \ref{t1.1}(ii) implies Theorem \ref{t1.1}(iii) and
Theorem \ref{t1.1}(iv) implies Theorem \ref{t1.1}(iii).
To this end, we need some fine estimates on the sharp maximal function $M^\sharp_r$ (see \eqref{e3.1} below)
and the non-centered doubling Hardy-Littlewood
maximal function $N_r$ (see \eqref{e3.2} below); for example, see the technical
Lemmas \ref{l3.1} and \ref{l3.2} concerning with the
operators $M^\sharp_r$ and $N_r$ from \cite{ly12}
and the estimate for $M^\sharp_r(\cm(f))$ in Lemma \ref{l3.4}.
We also need to consider the decomposition of the function $f$;
for example, in the proof that Theorem \ref{t1.1}(ii) implies Theorem \ref{t1.1}(iii),
for any fixed $\ell\in(0,\,\fz)$, we split $f$ into $f_1$ and $f_2$ with
$f_1:=f\chi_{\{y\in\cx:\ |f(y)|>\ell\}}$ and
$f_2:=f\chi_{\{y\in\cx:\ |f(y)|\le\ell\}}$, while in the proof that
Theorem \ref{t1.1}(iv) implies Theorem \ref{t1.1}(iii), we use
the Calder\'on-Zygmund decomposition from \cite[Theorem 6.3]{ad10} (see also Lemma \ref{l2.2} below).
Here and in what follows, for any $\mu$-measurable set
$E$, $\chi_E$ denotes its {\it characteristic function}.
Based on these facts, by some argument similar to that used in the proof of
\cite[Theorem 1.1]{ly12}, we then complete the proof of Theorem \ref{t1.1}.

Section \ref{s4} is devoted to the proof of Theorem \ref{t1.2}.
Indeed, Theorem \ref{t1.2}(i) can be deduced directly from Theorems \ref{t1.1} and \ref{t2.1}.
By using a technical estimate for $\cm$ (see \eqref{e4.3} below) and
some argument used in the proof of Theorem \ref{t2.1}, we then obtain the desired conclusion
of Theorem \ref{t1.2}(ii).

We finally make some conventions on notation. Throughout this paper,
we denote by $C$ a \emph{positive constant} which is independent of the
main parameters involved, but may vary from line to line. \emph{Positive
constants} with subscripts, such as $C_1$, do not change in different
occurrences. The subscripts of a constant indicate the parameters it
depends on. The {\it symbol} $Y\ls Z$ means that there exists a
positive constant $C$ such that $Y\le CZ$. The {\it symbol} $A\sim
B$ means that $A\ls B\ls A$. For any ball $B\subset\cx$, we denote
its {\it center} and {\it radius}, respectively, by $c_B$ and $r_B$
and, moreover, for any $\rho\in(1,\,\fz)$, the ball $B(c_B,\,\rho
r_B)$ by $\rho B$. Given any $q\in(1,\,\fz)$, let $q':=q/(q-1)$
denote its {\it conjugate index}. Also, let $\nn:=\{1,2,\ldots\}$.

\section{Boundedness of Marcinkiewicz integrals\label{s2}}

In this section, under the assumption that the
Marcinkiewicz integral is bounded on $L^{p_0}(\mu)$ for some
$p_0\in(1,\,\fz)$, we then obtain its boundedness on Lebesgue spaces
and Hardy spaces. We first recall the notions of
$(\az,\,\bz)$-doubling property and the space $\rblo(\mu)$.

\begin{definition}\label{d2.1}\rm
Let $\az,\,\bz\in(1,\,\fz)$. A ball $B:=B(x,\,r)\subset\cx$ is called
{\it $(\az,\,\bz)$-doubling}, if $\mu(\az B)\le\bz\mu(B)$.
\end{definition}

It was proved in \cite{h10} that, if a metric measure space
$(\cx,\,d,\,\mu)$ is upper doubling and
$\bz>C_\lz^{\log_2\az}=:\az^\nu$, then, for every ball
$B(x,\,r)\subset\cx$, there exists some $j\in\zz_+:=\nn\cup\{0\}$
such that $\az^j B$ is $(\az,\,\bz)$-doubling. Moreover, let
$(\cx,\,d)$ be geometrically doubling, $\bz>\az^n$ with $n:=\log_2
N_0$ and $\mu$ be a Borel measure on $\cx$ which is finite on bounded
sets. Hyt\"onen \cite{h10} also showed that, for $\mu$-almost every
$x\in\cx$, there exist arbitrarily small $(\az,\,\bz)$-doubling
balls centered at $x$. Furthermore, the radius of these balls may be
chosen to be of the form $\az^{-j}r$ for $j\in\nn$ and any
preassigned number $r\in(0,\,\fz)$. Throughout this paper, for any
$\az\in(1,\,\fz)$ and ball $B$, ${\wz B}^\az$ always denotes the {\it
smallest $(\az,\,\bz_\az)$-doubling ball of the form $\az^j B$ with
$j\in\nn$}, where
\begin{equation}\label{e2.1}
\bz_\az:=\max\lf\{\az^{3n},\,\az^{3\nu}\r\}+30^n+30^\nu
=\az^{3(\max\{n,\,\nu\})}+30^n+30^\nu.
\end{equation}
If $\az=6$, we {\it denote the ball ${\wz B}^\az$ simply by $\wz B$}.

The following space $\rblo(\mu)$ was introduced in \cite{ly11}.
Recall that the classical space ${\mathop\mathrm{BLO}}(\rn)$ was
introduced by Coifman and Rochberg \cite{cr80} and, in the setting of
$(\rn,\,|\cdot|,\,\mu)$ with $\mu$ only satisfying the polynomial
growth condition, the space $\rblo(\mu)$ was first introduced by
Jiang \cite{j05}.

\begin{definition}\label{d2.2}\rm
Let $\eta,\,\rho\in(1,\,\fz)$, and $\bz_\rho$ be as in \eqref{e2.1}.
A real-valued function $f\in L^1_{\loc}(\mu)$ is said to be in the
{\it space $\rblo(\mu)$}, if there exists a non-negative constant $C$
such that, for all balls $B$,
\begin{equation*}
\frac{1}{\mu(\eta B)}\int_B\lf[f(y)-{\mathop\einf_{{\wz
B}^\rho}}f\r]\,d\mu(y)\le C
\end{equation*}
and, for all $(\rho,\,\bz_\rho)$-doubling balls $B\subset S$,
\begin{equation*}
{\mathop\einf_B}f-{\mathop\einf_S}f\le C[1+\dz(B,\,S)].
\end{equation*}
Moreover, the $\rblo(\mu)$ {\it norm} of $f$ is defined to be the
minimal constant $C$ as above and denoted by $\|f\|_{\rblo(\mu)}$.
\end{definition}

It was proved in \cite{ly11} that $\rblo(\mu)\subset\rbmo(\mu)$ and
the definition of $\rblo(\mu)$ is independent of the choice of the
constants $\eta,\,\rho\in(1,\,\fz)$.

\begin{theorem}\label{t2.1}
Let $K$ satisfy \eqref{e1.5} and \eqref{e1.6}, and $\cm$ be as in
\eqref{e1.7}. Suppose that $\cm$ is bounded on $L^{p_0}(\mu)$ for
some $p_0\in(1,\,\fz)$. Then,
\begin{itemize}
\item[{\rm (i)}] $\cm$ is bounded from $L^1(\mu)$ into $L^{1,\,\fz}(\mu)$;

\item[{\rm (ii)}] $\cm$ is bounded from $H^1(\mu)$ into $L^1(\mu)$;

\item[{\rm (iii)}] for $f\in
L^\fz(\mu)$, $\cm(f)$ is either infinite everywhere or finite
$\mu$-almost everywhere; more precisely, if $\cm(f)$ is finite at
some point $x_0\in\cx$, then $\cm(f)$ is finite $\mu$-almost
everywhere and
$$
\|\cm(f)\|_{\rblo(\mu)}\le C\|f\|_{L^\fz(\mu)},
$$
where the positive constant $C$ is independent of $f$;

\item[{\rm (iv)}] $\cm$ is bounded on  $L^p(\mu)$ for all $p\in(1,\,\fz)$.
\end{itemize}
\end{theorem}

To prove Theorem \ref{t2.1}, we first recall some necessary
technical lemmas. The following useful properties of $\dz$ were
proved in \cite{hyy10}.

\begin{lemma}\label{l2.1}
{\rm (i)} For all balls $B\subset R\subset S$, it holds true that
$\dz(B,\,R)\le\dz(B,\,S)$.

{\rm (ii)} For any $\rho\in[1,\,\fz)$, there exists a positive
constant $C$, depending on $\rho$, such that, for all balls $B\subset
S$ with $r_S\le\rho r_B$, $\dz(B,\,S)\le C$.

{\rm (iii)} For any $\az\in(1,\,\fz)$, there exists a positive
constant $\wz C$, depending on $\az$, such that, for all balls $B$,
$\dz(B,\,{\wz B}^\az)\le {\wz C}$.

{\rm (iv)} There exists a positive constant $c$ such that, for all
balls $B\subset R\subset S$, $\dz(B,\,S)\le\dz(B,\,R)+c\dz(R,\,S)$.
In particular, if $B$ and $R$ are concentric, then $c=1$.

{\rm (v)} There exists a positive constant $\wz c$ such that, for all
balls $B\subset R\subset S$, $\dz(R,\,S)\le{\wz c}[1+\dz(B,\,S)]$;
moreover, if $B$ and $R$ are concentric, then
$\dz(R,\,S)\le\dz(B,\,S)$.
\end{lemma}

Now we recall the Calder\'on-Zygmund decomposition from
\cite[Theorem 6.3]{ad10}.

\begin{lemma}\label{l2.2}
Let $p\in[1,\,\fz)$, $f\in L^p(\mu)$ and $\ell\in(0,\fz)$
($\ell>\ell_0:=\gz^{\frac{1}{p}}_0[\mu(\cx)]^{-\frac{1}{p}}\|f\|_{L^p(\mu)}$ if
$\mu(\cx)<\fz$, where $\gz_0$ is any fixed positive constant
satisfying that $\gz_0>max\{C_\lz^{3\log_2 6},\,6^{3n}\}$, $C_\lz$ is
as in \eqref{e1.3} and $n:=\log_2 N_0$). Then,

{\rm (i)} there exists an almost disjoint family $\{6B_j\}_j$ of
balls such that $\{B_j\}_j$ is pairwise disjoint,
$$
\frac{1}{\mu(6^2B_j)}\int_{B_j}|f(x)|^p\,d\mu(x)
>\frac{\ell^p}{\gz_0} \quad{\text for\, all\, j,}
$$
$$
\frac{1}{\mu(6^2\eta B_j)}\int_{\eta B_j}|f(x)|^p\,d\mu(x) \le
\frac{\ell^p}{\gz_0} \quad{\text for\ all\ j\ and\ all\
\eta\in(2,\,\fz)}
$$
and
$$
|f(x)|\le \ell \quad{\text for\ \mu-almost\ every\ x\in
\cx\setminus(\cup_j6B_j);}
$$

{\rm (ii)} for each j, let $S_j$ be a $(3\times
6^2,\,C_\lz^{\log_2(3\times 6^2)+1})$-doubling ball of the family
$\{(3\times 6^2)^kB_j\}_{k\in\nn}$ and
$\oz_j:=\chi_{6B_j}/(\sum_k\chi_{6B_k})$. Then, there exists a family
$\{\fai_j\}_j$ of functions such that, for each $j$,
$\supp(\fai_j)\subset S_j$, $\fai_j$ has a constant sign on $S_j$,
$$
\int_\cx\fai_j(x)\,d\mu(x)=\int_{6B_j}f(x)\oz_j(x)\,d\mu(x),
$$
$$
\sum_j|\fai_j(x)|\le \gz\ell \quad{\text for\ \mu-almost\ every\
x\in\cx,}
$$
where $\gz$ is some positive constant, depending only on
$(\cx,\,\mu)$, and there exists a positive constant $C$, independent
of $f$, $\ell$ and $j$, such that, when $p=1$, it holds true that
$$
\|\fai_j\|_{L^\fz (\mu)}\mu(S_j)\le
C\int_\cx|f(x)\oz_j(x)|\,d\mu(x)
$$
and, when $p\in(1,\,\fz)$, it holds true that
\begin{equation*}
\lf[\int_{S_j}|\fai_j(x)|^p\,d\mu(x)\r]^{1/p}[\mu(S_j)]^{1/p'}\le
\frac{C}{\ell^{p-1}}\int_\cx |f(x)\oz_j(x)|^p\,d\mu(x);
\end{equation*}

{\rm (iii)} for $p\in(1,\,\fz)$, if, for any $j$, choosing $S_j$ in
{\rm(ii)} to be the smallest $(3\times6^2,\,C_\lz^{\log_2(3\times
6^2)+1})$-doubling ball of the family $\{(3\times 6^2)^kB_j\}_{k\in\nn}$, then
$h:=\sum_j(f\oz_j-\fai_j)\in H_{\rm atb}^{1,\,p}(\mu)$ and there exists
a positive constant $C$, independent of $f$ and $\ell$, such that
$$
\|h\|_{H_{\rm atb}^{1,\,p}}(\mu)\le
\frac{C}{\ell^{p-1}}\|f\|_{L^p(\mu)}^p.
$$
\end{lemma}

The following characterization of the space $\rblo(\mu)$ was proved
in \cite{ly11}.

\begin{lemma}\label{l2.3}
Let $\rho\in(1,\,\fz)$ and $\bz_\rho$ be as in \eqref{e2.1}. If
$f\in\rblo(\mu)$, then there exists a non-negative constant $C_1$
satisfying that, for all $(\rho,\,\bz_\rho)$-doubling balls $B$,
\begin{equation*}
\frac{1}{\mu(B)}\int_B\lf[f(y)-{\mathop\einf_B} f\r]\,d\mu(y)\le
C_1
\end{equation*}
and, for all $(\rho,\,\bz_\rho)$-doubling balls $B\subset S$,
\begin{equation}\label{e2.2}
m_B(f)-m_S(f)\le C_1[1+\dz(B,\,S)],
\end{equation}
where above and in what follows, $m_B(f)$ denotes the mean of $f$ over
$B$, namely,
$$m_B(f):=\frac{1}{\mu(B)}\int_B f(y)\,d\mu(y).$$
Moreover, the minimal constant $C_1$ is equivalent to
$\|f\|_{\rblo(\mu)}$.
\end{lemma}

To prove Theorem \ref{t1.1}, we also need the following two
interpolation results.

\begin{lemma}\label{l2.4}
Let $p\in(1,\,\fz)$ and $T$ be a sublinear operator bounded from
$L^1(\mu)$ to $L^{1,\,\fz}(\mu)$. If there exists a positive
constant $C$ such that, for all $(p,\,1)_{\lz}-$ atomic blocks $b$,
\begin{equation*}
\|Tb\|_{L^1(\mu)}\le C|b|_{H^{1,\,p}_{\rm atb}(\mu)},
\end{equation*}
then $T$ extends to be a bounded sublinear operator from $H^1(\mu)$
to $L^1(\mu)$.
\end{lemma}

The proof of Lemma \ref{l2.4} is similar to that of \cite[Theorem
1.13]{yy11}, the details being omitted. The following lemma is just
\cite[Theorem 1.1]{ly12}.

\begin{lemma}\label{l2.5}
Suppose that $T$ is a sublinear operator bounded from $L^\fz(\mu)$
to $\rbmo(\mu)$ and from $H^1(\mu)$ to $L^{1,\,\fz}(\mu)$. Then $T$
extends boundedly to $L^p(\mu)$ for every $p\in(1,\,\fz)$.
\end{lemma}

Based on the above lemmas, we now turn to the proof of Theorem
\ref{t2.1}.

\begin{proof}[Proof of Theorem \ref{t2.1}]
We first show (i). Let $f\in L^1(\mu)$ and $\ell\in(0,\,\fz)$. To prove (i),
it suffices to show that
\begin{equation}\label{e2.3}
\mu\lf(\lf\{x\in\cx:\
\cm(f)(x)>\ell\r\}\r)\ls\ell^{-1}\|f\|_{L^1(\mu)}.
\end{equation}
By applying Lemma \ref{l2.2} and its notation, we see that $f=g+h$, where
$h:=\sum_j(f\oz_j-\fai_j)=:\sum_jh_j$. Obviously,
$\|g\|_{L^\fz(\mu)}\ls\ell$ and $\|g\|_{L^1(\mu)}\ls
\|f\|_{L^1(\mu)}$. This, together with the $L^{p_0}(\mu)$
boundedness of $\cm$, implies that
\begin{equation*}
\mu(\{x\in\cx:\ \cm(g)(x)>\ell\})\ls
\ell^{-p_0}\|g\|^{p_0}_{L^{p_0}(\mu)}\ls \ell^{-1}\|f\|_{L^1(\mu)}.
\end{equation*}
By Lemma \ref{l2.2}(i), to prove \eqref{e2.3}, we only need to prove
that
\begin{equation}\label{e2.4}
\mu\lf(\lf\{x\in\cx\setminus\lf(\bigcup_j6^2B_j\r):\ \cm(h)(x)>\ell\r\}\r)\ls
\ell^{-1}\int_\cx|f(x)|\,d\mu(x).
\end{equation}

To this end, for each fixed
$j$, let $S_j$ be as in Lemma \ref{l2.2}(iii)
with $c_{S_j}$ and $r_{S_j}$ being, respectively, its center and radius, and write
\begin{align*}
&\int_{\cx\setminus 2S_j}\cm(h_j)(x)\,d\mu(x)\\
&\hs\le \int_{\cx\setminus 2S_j}\lf[\int_0^{d(x,\,c_{S_j})+r_{S_j}}
\lf|\int_{d(x,\,y)< t}K(x,\,y)h_j(y)\,d\mu(y)\r|^2\,\frac{dt}{t^3}\r]^{1/2}\,d\mu(x)\\
&\hs\hs+\int_{\cx\setminus
2S_j}\lf[\int_{d(x,\,c_{S_j})+r_{S_j}}^{\fz}
\cdots\r]^{1/2}\,d\mu(x)=:{\rm I_1+I_2},
\end{align*}
From \eqref{e1.3} and \eqref{e1.4}, we deduce that, for any ball
$B$ with the center $c_B$, $x\not\in kB$ with $k\in(1,\,\fz)$ and $y\in B$,
\begin{equation}\label{e2.5}
\lz(c_B,\,d(x,\,c_B))\sim \lz(x,\,d(x,\,c_B))\sim\lz(x,\,d(x,\,y)),
\end{equation}
which, together with the Minkowski inequality, \eqref{e1.3} and
\eqref{e1.5}, shows that
\begin{align*}
{\rm I_1} &\ls\int_{\cx\setminus
2S_j}\int_{\cx}\lf[\int_{d(x,\,y)}^{d(x,\,c_{S_j})+r_{S_j}}\,\frac{dt}{t^3}\r]^{1/2}
\frac{|h_j(y)|d(x,\,y)}{\lz(x,\,d(x,\,y))}\,d\mu(y)\,d\mu(x)\\
&\ls r_{S_j}^{1/2}\lf[\int_{\cx}|h_j(y)|\,d\mu(y)\r]
\int_{\cx\setminus
2S_j}\frac{1}{[d(x,\,c_{S_j})]^{1/2}\lz(x,\,d(x,\,c_{S_j}))}\,d\mu(x)
\ls\|h_j\|_{L^1(\mu)}.
\end{align*}
For $x\in\cx\setminus 2S_j$ and $y\in S_j$, it holds true that $d(x,\,y)<
d(x,\,c_{S_j})+r_{S_j}$. Thus, by the vanishing moment of $h_j$ and
\eqref{e1.6}, we obtain
\begin{align*}
{\rm I_2}&\ls\int_{\cx\setminus
2S_j}\lf|\int_{\cx}[K(x,\,y)-K(x,c_{S_j})]h_j(y)\,d\mu(y)\r|
\frac{1}{d(x,\,c_{S_j})+r_{S_j}}\,d\mu(x)\\
&\ls\int_{\cx}|h_j(y)|\int_{\cx\setminus
2S_j}|K(x,\,y)-K(x,c_{S_j})|
\frac{1}{d(x,\,c_{S_j})}\,d\mu(x)\,d\mu(y) \ls\|h_j\|_{L^1(\mu)}.
\end{align*}

Notice that $\supp(f\oz_j)\subset 6B_j$ and $|\oz_j|\le 1$. From
this, the Minkowski inequality, \eqref{e1.5}, \eqref{e2.5} and
Lemma \ref{l2.1}, it follows that
\begin{align}\label{e2.6}
&\int_{(2S_j)\setminus6^2B_j}\cm(f\oz_j)(x)\,d\mu(x)\nonumber\\
&\hs\ls\int_{(2S_j)\setminus6^2B_j}\int_{\cx}
\lf[\int_{d(x,\,y)}^{\fz}\frac{dt}{t^3}\r]^{1/2}
\frac{|f(y)\oz_j(y)|d(x,\,y)}{\lz(x,\,d(x,\,y))}\,d\mu(y)d\mu(x)\nonumber\\
&\hs\ls\int_{(2S_j)\setminus6^2B_j}\frac{1}{\lz(c_{B_j},\,d(x,\,c_{B_j}))}\,d\mu(x)
\int_{6B_j}|f(y)|\,d\mu(y)\nonumber\\
&\hs\ls\dz(B_j,S_j)\int_{6B_j}|f(y)|\,d\mu(y)
\ls\int_{6B_j}|f(y)|\,d\mu(y).
\end{align}
On the other hand,  by the H\"older inequality, the
$L^{p_0}(\mu)$-boundedness of $\cm$ and Lemma \ref{l2.2}(ii), we
conclude that
\begin{align*}
\int_{2S_j}\cm(\fai_j)(x)\,d\mu(x)
&\le\lf\{\int_{2S_j}[\cm(\fai_j)(x)]^{p_0}\,d\mu(x)\r\}^{1/p_0}
\lf[\mu(2S_j)\r]^{1/p'_0}\\
&\ls\lf[\int_{S_j}|\fai_j(x)|^{p_0}\,d\mu(x)\r]^{1/p_0}\lf[\mu(S_j)\r]^{1/p'_0}
\ls\int_{6B_j}|f(y)|\,d\mu(y),
\end{align*}
where $1/p_0+1/p'_0=1$. The above two estimates, together with the estimates for ${\rm I}_1$
and ${\rm I}_2$ and Lemma \ref{l2.2}, show that
\begin{align*}
&\mu\lf(\lf\{x\in\cx\setminus\lf(\bigcup_j6^2B_j\r):\ \cm(h)(x)>\ell\r\}\r)\\
&\hs\le \ell^{-1}\lf[\sum_j\int_{\cx\setminus
2S_j}\cm(h_j)(x)\,d\mu(x)+ \sum_j\int_{(2S_j)\setminus
6^2B_j}\cdots\r]\\
&\hs\ls\ell^{-1}\lf[\sum_j\|h_j\|_{L^1(\mu)}+\sum_j\int_{6B_j}|f(y)|\,d\mu(y)\r]
\ls \ell^{-1}\int_\cx|f(x)|\,d\mu(x),
\end{align*}
which implies \eqref{e2.4} and hence completes the proof of (i).

To prove (ii), as pointed out in Remark \ref{r1.3}, since the definition of
$H^1(\mu)$ is independent of the choice of the constant $\rho\in
(1,\,\fz)$, without loss of generality, we may assume that $\rho=2$ in Definition \ref{d1.4}. It
follows, from (i), that $\cm$ is bounded from $L^1(\mu)$ to
$L^{1,\fz}(\mu)$. Thus, by Lemma \ref{l2.4}, to show (ii), it suffices to
prove that, for all $(p_0,\,1)_\lz$-atomic blocks $b$,
\begin{equation}\label{e2.7}
\|\cm(b)\|_{L^1(\mu)}\ls|b|_{H^{1,\,p_0}_{\rm atb}(\mu)}.
\end{equation}
Let $b:=\sum^2_{j=1}\kz_j a_j$ be a $(p_0,\,1)_\lz$-atomic block, where,
for any $j\in\{1,\,2\}$, $\supp(a_j)\subset B_j\subset B$ for some
$B_j$ and $B$ as in Definition \ref{d1.4}. Write
\begin{align*}
&\int_{\cx}\cm(b)(x)\,d\mu(x)\\
&\hs=\int_{\cx\setminus
2B}\cm(b)(x)\,d\mu(x)+\int_{2B}\cdots\\
&\hs\le\int_{\cx\setminus
2B}\cm(b)(x)\,d\mu(x)+\sum_{j=1}^2|\kz_j|\lf[\int_{2B_j}\cm(a_j)(x)\,d\mu(x)
+\int_{(2B)\setminus2B_j}\cdots\r]\\
&\hs=:{\rm
J_1 + J_2}.
\end{align*}

By Definition \ref{d1.4} and an argument similar to that used in the estimates
for ${\rm I}_1$ and ${\rm I}_2$, we see that
\begin{equation}\label{e2.8}
{\rm J_1}\ls \|b\|_{L^1(\mu)}\ls|b|_{H^{1,\,p_0}_{\rm atb}(\mu)}.
\end{equation}

From the H\"older inequality, the $L^{p_0}(\mu)$ boundedness of
$\cm$ and Definition \ref{d1.4}(iii), it follows that, for each fixed
$j$,
\begin{align}\label{e2.9}
\int_{2B_j}\cm(a_j)(x)\,d\mu(x)&\le\|\cm(a_j)\|_{L^{p_0}(\mu)}
[\mu(2B_j)]^{1/p'_0}\nonumber\\
&\ls\|a_j\|_{L^{p_0}(\mu)}[\mu(2B_j)]^{1/p'_0}\ls 1.
\end{align}
Similar to the estimate for \eqref{e2.6}, by Definition
\ref{d1.4}(iii), we have
\begin{align*}
\int_{(2S)\setminus2B_j}\cm(a_j)(x)\,d\mu(x)
&\ls\int_{(2S)\setminus2B_j}\frac{1}{\lz(c_{B_j},\,d(x,\,c_{B_j}))}\,d\mu(x)\|a_j\|_{L^1(\mu)}\\
&\ls\dz(B_j,\,S)\|a_j\|_{L^1(\mu)}\ls 1.
\end{align*}
Combining the above estimates, we see that
$$
{\rm J_2}\ls|\kz_1|+|\kz_2|\sim |b|_{H^{1,\,p_0}_{\rm atb}(\mu)},
$$
which, together with \eqref{e2.8}, implies \eqref{e2.7} and hence
completes the proof of (ii).

We now prove (iii). First, we claim that there exists a positive
constant $C$ such that, for any $f\in L^\fz(\mu)$ and
$(6,\,\bz_6)$-doubling ball $B$,
\begin{equation}\label{e2.10}
\frac{1}{\mu(B)}\int_B\cm(f)(y)\,d\mu(y)\le
C\|f\|_{L^\fz(\mu)}+\inf_{y\in B}\cm(f)(y).
\end{equation}
To prove this, we decompose $f$ as
$$
f(x)=f\chi_{5B}+f\chi_{\cx\setminus 5B}=: f_1+f_2.
$$
By the H\"older inequality and $L^{p_0}(\mu)$ boundedness of $\cm$,
we have
\begin{align}\label{e2.11}
\frac{1}{\mu(B)}\int_B\cm(f_1)(y)\,d\mu(y)
&\le\frac{1}{[\mu(B)]^{1/p_0}}
\lf\{\int_{\cx}[\cm(f_1)(y)]^{p_0}\,d\mu(y)\r\}^{1/p_0}\nonumber\\
&\ls\frac{[\mu(5B)]^{1/p_0}}{[\mu(B)]^{1/p_0}}\|f\|_{L^\fz(\mu)}\ls\|f\|_{L^\fz(\mu)}.
\end{align}
Noticing that, for $y\in B$ and $z\in\cx\setminus 5B$, it holds true that $d(y,z)> r_B$.
By the Minkowski inequality, \eqref{e1.3} and \eqref{e1.5}, we conclude
that, for any $y\in B$,
\begin{align}\label{e2.12}
\cm(f_2)(y) &=\lf[\int_{r_B}^\fz\lf|\int_{d(y,\,z)<
t}K(y,\,z)f_2(z)\,d\mu(z)\r|^{2}
\frac{dt}{t^3}\r]^{1/2}\nonumber\\
&\le\lf[\int_{r_B}^\fz\lf|\int_{d(y,\,z)<
t}K(y,\,z)f(z)\,d\mu(z)\r|^{2}
\frac{dt}{t^3}\r]^{1/2}\nonumber\\
&\hs+\lf[\int_{r_B}^\fz\lf|\int_{d(y,\,z)<
t}K(y,\,z)f_1(z)\,d\mu(z)\r|^{2}
\frac{dt}{t^3}\r]^{1/2}\nonumber\\
&\le\cm(f)(y)+\lf[\int_{r_B}^\fz\lf|\int_{d(y,\,z)<
6r_B}K(y,\,z)f_1(z)\,d\mu(z)\r|^{2}
\frac{dt}{t^3}\r]^{1/2}\nonumber\\
&\le \cm(f)(y)+C\|f\|_{L^\fz(\mu)}r_B^{-1}\int_{d(y,\,z)< 6r_B}
\frac{d(y,\,z)}{\lz(y,\,d(y,\,z))}\,d\mu(z)\nonumber\\
&\le \cm(f)(y)+C\|f\|_{L^\fz(\mu)},
\end{align}
where $C$ is a positive constant independent of $f$ and $y$. Thus,
the proof of the estimate \eqref{e2.10} can be reduced to proving
that, for all $x,\,y\in B$,
\begin{equation}\label{e2.13}
|\cm(f_2)(x)-\cm(f_2)(y)|\ls\|f\|_{L^\fz(\mu)}.
\end{equation}
To this end, write
\begin{align*}
&|\cm(f_2)(x)-\cm(f_2)(y)|\\
&\hs\le\lf[\int_0^\fz\lf|\int_{d(x,\,z)<
t}K(x,\,z)f_2(z)\,d\mu(z)
-\int_{d(y,\,z)< t}K(y,\,z)f_2(z)\,d\mu(z)\r|^2\,\frac{dt}{t^3}\r]^{1/2}\\
&\hs\le\lf\{\int_0^\fz\lf[\int_{d(y,\,z)< t\le d(x,\,z)}|K(y,\,z)||f_2(z)|\,d\mu(z)\r]^2\,\frac{dt}{t^3}\r\}^{1/2}\\
&\hs\hs+\lf\{\int_0^\fz\lf[\int_{d(x,\,z)< t\le d(y,\,z)}|K(x,\,z)||f_2(z)|\,d\mu(z)\r]^2\,\frac{dt}{t^3}\r\}^{1/2}\\
&\hs\hs+\lf\{\int_0^\fz\lf[\int_{\max\{d(y,\,z),\,d(x,\,z)\}< t}
|K(y,\,z)-K(x,\,z)||f_2(z)|\,d\mu(z)\r]^2\,\frac{dt}{t^3}\r\}^{1/2}\\
&\hs=:{\rm M}_1+{\rm M}_2+{\rm M}_3.
\end{align*}
Applying the Minkowski inequality, \eqref{e1.3}, \eqref{e1.5} and
\eqref{e2.5}, we conclude that, for all $x,\,y\in B$,
\begin{align*}
{\rm M}_1 &\ls\int_{\cx}\frac{|f_2(z)|d(y,\,z)}{\lz(y,\,d(y,\,z))}
\lf[\int_{d(y,\,z)< t\le d(x,\,z)}\frac{dt}{t^3}\r]^{1/2}\,d\mu(z)\\
&\ls\|f\|_{L^\fz(\mu)}\int_{\cx\setminus 5B}
\frac{{r_B}^{1/2}}{[d(z,\,c_B)]^{1/2}\lz({c_B,\,d(z,\,c_B))}}\,d\mu(z)
\ls\|f\|_{L^\fz(\mu)}.
\end{align*}
Similarly, ${\rm M}_2\ls \|f\|_{L^\fz(\mu)}$. Another application of
the Minkowski inequality and \eqref{e1.6} shows that
\begin{align*}
{\rm M}_3 &\ls\int_\cx|K(y,\,z)-K(x,\,z)||f_2(z)|
\lf[\int_{\max\{d(y,\,z),\,d(x,\,z)\}< t}\frac{dt}{t^3}\r]^{1/2}\,d\mu(z)\\
&\ls\|f\|_{L^\fz(\mu)}\int_{\cx\setminus 5B}
\frac{|K(y,\,z)-K(x,\,z)|}{d(z,\,c_B)}\,d\mu(z)
\ls\|f\|_{L^\fz(\mu)}.
\end{align*}
Combining the estimates for ${\rm M}_1$, ${\rm M}_2$ and ${\rm M}_3$, we obtain \eqref{e2.13}.
Thus, \eqref{e2.10} holds true.

By \eqref{e2.10}, for $f\in L^\fz(\mu)$, if $\cm(f)(x_0)<\fz$ for
some point $x_0\in \cx$, then $\cm(f)$ is finite $\mu$-almost
everywhere and, in this case,
$$
\frac{1}{\mu(B)}\int_B\lf[\cm(f)(x)-\mathop\einf_{x\in
B}\cm(f)(x)\r]\,d\mu(x)\ls\|f\|_{L^\fz(\mu)},
$$
provided that $B$ is a $(6,\,\bz_6)$-doubling ball. To prove that
$\cm(f)\in \rblo(\mu)$, by Lemma \ref{l2.3}, we still need to prove
that $\cm(f)$ satisfies \eqref{e2.2}. Let $B\subset S$ be any two
$(6,\,\bz_6)$-doubling balls. For any $x\in B$ and $y\in S$, we write
\begin{align*}
\cm(f)(x)
&\ls\cm(f\chi_{5B})(x)+\cm(f\chi_{(5S)\setminus 5B})(x)\\
&\hs+\lf[\cm(f\chi_{\cx\setminus 5S})(x)-\cm(f\chi_{\cx\setminus
5S})(y)\r]+\cm(f\chi_{\cx\setminus 5S})(y).
\end{align*}
By an estimate similar to that of \eqref{e2.12}, for all $y\in S$,
we have
$$
\cm(f\chi_{\cx\setminus 5S})(y)\le \cm(f)(y)+C\|f\|_{L^\fz(\mu)},
$$
where $C$ is a positive constant independent of $f$ and $y$. On the
other hand, by the estimate same as that of \eqref{e2.13}, for all
$x,\,y\in S$, we see that
$$
|\cm(f\chi_{\cx\setminus 5S})(x)-\cm(f\chi_{\cx\setminus 5S})(y)|\ls
\|f\|_{L^\fz(\mu)}.
$$
For all $x\in B$, by the Minkowski inequality, \eqref{e1.5},
\eqref{e2.5} and Lemma \ref{l2.1}, we obtain
\begin{align}\label{e2.14}
\cm(f\chi_{(5S)\setminus 5B})(x)
&\ls\int_{\cx}\lf[\int_{d(x,\,y)}^{\fz}\frac{dt}{t^3}\r]^{1/2}
\frac{|f\chi_{(5S)\setminus 5B}(y)|d(x,\,y)}{\lz(x,\,d(x,\,y))}\,d\mu(y)\nonumber\\
&\ls\|f\|_{L^{\fz}(\mu)}\int_{(5S)\setminus B}
\frac{1}{\lz(c_B,\,d(z,\,c_B))}\,d\mu(z)\nonumber\\
&\ls[1+\dz(B,\,S)]\|f\|_{L^{\fz}(\mu)}.
\end{align}
Therefore, for any $x\in B$ and $y\in S$, we find that
$$
\cm(f)(x)\ls\cm(f\chi_{5B})(x)+\cm(f)(y)+[1+\dz(B,\,S)]\|f\|_{L^{\fz}(\mu)}.
$$
Taking mean value over $B$ for $x$ and over $S$ for $y$, we conclude that
$$
m_B(\cm(f))-m_S(\cm(f))\ls [1+\dz(B,\,S)]\|f\|_{L^{\fz}(\mu)},
$$
where we used \eqref{e2.11}. This finishes the proof of Theorem \ref{t2.1}(iii).

Notice that $\rblo(\mu)\subset\rbmo(\mu)$. It then follows, from
(ii), (iii) and Lemma \ref{l2.5}, that $\cm$ is bounded on $L^p(\mu)$
for all $p\in(1,\,\fz)$, which implies (iv) and hence completes the
proof of Theorem \ref{t2.1}.
\end{proof}

\section{Proof of Theorem \ref{t1.1}\label{s3}}

To prove Theorem \ref{t1.1}, we need some maximal
functions in \cite{h10,ad10} as follows. Let $f\in L_\loc^1(\mu)$
and $x\in\cx$. The {\it non-centered doubling Hardy-Littlewood
maximal function} $N(f)(x)$ and the {\it sharp maximal function}
$M^\sharp(f)(x)$ are, respectively, defined by setting,
$$
N(f)(x):=\sup_{\gfz{B\ni x}{B\ (6, \,\bz_6)-{\rm
doubling}}}\dfrac{1}{\mu(B)}\int_B|f(y)|\,d\mu(y)
$$
and
\begin{align*}
M^\sharp(f)(x)&:=\sup_{B\ni x}
\frac{1}{\mu(5B)}\int_B|f(y)-m_{\wz B}(f)|\,d\mu(y)\nonumber\\
&\hs+\sup_{\gfz{x\in B\subset S}{B,\, S\ (6,\, \bz_6)-{\rm
doubling}}}\dfrac{|m_B(f)-m_S(f)|}{1+\dz(B,\,S)}.
\end{align*}
Moreover, for all
$r\in(0,\,\fz)$, the operators $M^\sharp_r$ and $N_r$ are defined, respectively, by
setting, for all $f\in L^r_\loc(\mu)$ and $x\in\cx$,
\begin{equation}\label{e3.1}
M^\sharp_r(f)(x):=\{M^\sharp(|f|^r)(x)\}^{\frac{1}{r}}
\end{equation}
and
\begin{equation}\label{e3.2}
N_r(f):= [N(|f|^r)]^{1/r}.
\end{equation}
By the Lebesgue differentiation theorem, we see that,
for $\mu$-almost every $x\in\cx$,
\begin{equation}\label{e3.3}
|f(x)|\le N(f)(x);
\end{equation}
see \cite[Corollary 3.6]{h10}. Moreover, it follows, from
\cite[Proposition 3.5]{h10}, that, for any $p\in[1,\,\fz)$, $Nf$ is
bounded from $L^p(\mu)$ to $L^{p,\,\fz}(\mu)$.

The following two technical lemmas were, respectively, \cite[Lemmas 3.2 and 3.3]{ly12}.

\begin{lemma}\label{l3.1}
Let $p\in[1,\,\fz)$ and $f\in L^1_{\rm loc}(\mu)$ such that
$\int_\cx f(x)\,d\mu(x)=0$ if $\mu(\cx)<\fz$. If, for any $R\in(0,\,\fz)$,
$$
\sup_{\ell\in(0,\,R)}\ell^p\mu(\{x\in\cx:\ N(f)(x)>\ell\})<\fz,
$$
then there exists a positive constant $C$, independent of $f$, such that
\begin{equation*}
\sup_{\ell\in(0,\,\fz)}\ell^p\mu(\{x\in\cx:\ N(f)(x)>\ell\})
\le C\sup_{\ell\in(0,\,\fz)}\ell^p\mu\lf(\lf\{x\in\cx:\ M^\sharp
(f)(x)>\ell\r\}\r).
\end{equation*}
\end{lemma}

\begin{lemma}\label{l3.2}
Let $r\in(0,\,1)$ and $N_r(f)$ be as in \eqref{e3.2}. Then, for any
$p\in[1,\,\fz)$, there exists a positive constant $C$, depending on
$r$, such that, for any suitable function $f$ and $\ell\in(0,\,\fz)$,
$$
\mu(\{x\in\cx:\ N_r(f)(x)>\ell\})\le C\ell^{-p}
\sup_{\tau\in[\ell,\,\fz)}\tau^p\mu(\{x\in\cx:\ |f(x)|>\tau\}).
$$
\end{lemma}

\begin{lemma}\label{l3.3}
Let $r\in(0,\,1)$, $K$ satisfy \eqref{e1.5} and \eqref{e1.6}, and
$\cm$ be as in \eqref{e1.7}. If $\cm$ is bounded from $H^1(\mu)$
into $L^1(\mu)$, then there exists a positive constant $C$,
depending on $r$, such that, for any $\rho\in(1,\,\fz)$, ball
$B$ and function $a\in L^\fz(\mu)$ supported on $B$,
\begin{equation}\label{e3.4}
\dfrac{1}{\mu(\rho B)}\dint_B[\cm(a)(x)]^r\,d\mu(x)\le
C\|a\|^r_{L^\fz(\mu)}.
\end{equation}
\end{lemma}

\begin{proof}
Without loss of generality, we may assume that $\rho=2$. For any given
ball $B:=B(c_B,\,r_B)$, we consider the following two cases on $r_B$.

{\it Case} (i) $r_B\le\diam(\supp\mu)/40$. We use the same notation
as in the proof of \cite[Lemma 3.1]{lyy}. Let $S$ be the {\it
smallest ball} of the form $6^jB$ such that $\mu(6^jB\setminus
2B)>0$ with $j\in\nn$. Thus, $\mu(6^{-1}S\setminus 2B)=0$ and
$\mu(S\setminus 2B)>0$. This leads to $\mu(S\setminus(6^{-1}S\cup
2B))>0$ and $\wz B\subset \wz S$. By this and \cite[Lemma 3.3]{h10},
we choose $x_0\in S\setminus(6^{-1}S\cup 2B)$ such that the ball
center at $x_0$ with the radius $6^{-k}r_S$ for some integer $k\ge
2$ is $(6,\,\bz_6)$-doubling. Let $B_0$ be the {\it largest ball} of
this form. Then, it is easy to show that $B_0\subset 2S$ and
$d(B_0,\, B)\ge r_B/2$. It was proved, in the proof of \cite[Lemma
3.1]{lyy}, that $\dz(B,\,2S)\ls 1$ and $\dz(B_0,\,2S)\ls 1$, which
imply that $\dz(B,\,\wz{2S})\ls 1$ and $\dz(B_0,\,\wz{2S})\ls 1$.

For any $a\in L^\fz(\mu)$ supported on $B$, set
\begin{equation*}
C_{B_0}:=-\frac{1}{\mu(B_0)}\int_{\cx}a(x)\,d\mu(x)\quad{\rm
and}\quad b:=a+C_{B_0}\chi_{B_0}.
\end{equation*}
It is easy to see that $b$ is an $(\fz,\,1)_\lz$-atomic block with $\supp(b)\subset
2S$ and $\int_\cx b(x)\,d\mu(x)=0$. Moreover, by the choice of
$C_{B_0}$, the doubling property of $B_0$ and the assumption of $a$,
we have
\begin{equation}\label{e3.5}
|C_{B_0}|\mu(2B_0)\ls|C_{B_0}|\mu(B_0)\ls\|a\|_{L^1(\mu)}\ls\|a\|_{L^\fz(\mu)}\mu(2B),
\end{equation}
which further shows that
\begin{equation*}
\|b\|_{H^{1,\,\fz}_{\rm atb}(\mu)}
\le[1+\dz(B,\,2S)]\|a\|_{L^\fz(\mu)}\mu(2B)
+[1+\dz(B_0,\,2S)]|C_{B_0}|\mu(2B_0)\ls\|a\|_{L^\fz(\mu)}\mu(2B).
\end{equation*}

Notice that, for any $x\in B$ and $y\in B_0$, it holds true that $d(x,\,y)\ge r_B/2$. It
then follows, from the Minkowski inequality, \eqref{e1.5},
\eqref{e1.3}, \eqref{e1.4} and \eqref{e3.5}, that, for any $x\in B$,
\begin{align*}
\cm(C_{B_0}\chi_{B_0})(x)&=\lf[\int_0^\fz\lf|\int_{d(x,\,y)<
t}K(x,\,y)C_{B_0}\chi_{B_0}(y)\,d\mu(y)\r|^2\,\frac{dt}{t^3}\r]^{1/2}\\
&\ls|C_{B_0}|\int_{B_0}\lf[\int_{d(x,\,y)}^\fz\frac{dt}{t^3}\r]^{1/2}
\frac{d(x,\,y)}{\lz(x,\,d(x,\,y))}d\mu(y)\\
&\ls|C_{B_0}|\frac{\mu(B_0)}{\lz(x,\,r_B)}
\ls\frac{\|a\|_{L^\fz(\mu)}\mu(2B)}{\lz(c_B,\,r_B)}\ls\|a\|_{L^\fz(\mu)}.
\end{align*}
On the other hand, by the boundedness from $H^1(\mu)$ to $L^1(\mu)$
of $\cm$, together with the H\"older inequality, we see that, for any
$r\in(0,\,1)$, there exists a positive constant $C_r$, depending on
$r$, such that, for all $b\in H^1(\mu)$ and balls $B$,
\begin{equation*}
\int_B[\cm(b)(x)]^r\,d\mu(x)\le
C_r\frac{\|b\|^r_{H^1(\mu)}}{[\mu(B)]^{r-1}}.
\end{equation*}
Combining the above estimates, we see that
\begin{align}\label{e3.6}
\int_B[\cm(a)(x)]^r\,d\mu(x)&\le\int_B[\cm(b)(x)]^r\,d\mu(x)
+\int_B[\cm(C_{B_0}\chi_{B_0})]^r\,d\mu(x)\nonumber\\
&\ls\frac{\|b\|^r_{H^1(\mu)}}{[\mu(B)]^{r-1}}+\mu(B)\|a\|^r_{L^\fz(\mu)}
\ls\mu(2B)\|a\|^r_{L^\fz(\mu)}.
\end{align}

{\it Case} (ii) $r_B>\diam(\supp\mu)/40$. In this case, without loss
of generality, we may assume that $r_B\le8\diam(\supp\mu)$. Then Remark
\ref{r1.2}(ii) tells us that $B\cap\supp\mu$ is covered by finite
number balls $\{B_j\}_{j=1}^N$ with radius $r_B/800$, where
$N\in\nn$. For $j\in\{1,\,\ldots,\,N\}$ and $a$ as Lemma \ref{l3.3},
we define $a_j:=\frac{\chi_{B_j}}{\sum_{k=1}^N\chi_{B_k}}a$.

By the argument used in Case (i), we see that \eqref{e3.6} also
holds true, if we replace $B$ and $a$ by $2B_j$ and $a_j$, respectively.
It then follows that
\begin{align*}
\int_B[\cm(a)(x)]^r\,d\mu(x)
&\le\sum_{j=1}^N\int_{2B_j}[\cm(a_j)(x)]^r\,d\mu(x)\\
&\ls\sum_{j=1}^N\mu(4B_j)\|a_j\|^r_{L^\fz(\mu)}\ls\mu(2B)\|a\|^r_{L^\fz(\mu)},
\end{align*}
which, combined with \eqref{e3.6}, completes the proof of Lemma
\ref{l3.3}.
\end{proof}

\begin{lemma}\label{l3.4}
Let $r\in(0,\,1)$, $K$ satisfy \eqref{e1.5} and \eqref{e1.6}, and
$\cm$ be as in \eqref{e1.7}.  Suppose that $\cm$ is bounded from
$H^1(\mu)$ into $L^1(\mu)$, or from $L^1(\mu)$ into
$L^{1,\,\fz}(\mu)$. Then, there exists a positive constant $C_r$,
depending on $r$, such that, for all $f\in L^\fz_b(\mu)$,
\begin{equation}\label{e3.7}
\|M^\sharp_r(\cm(f))\|_{L^\fz(\mu)}\le C_r\|f\|_{L^\fz(\mu)}.
\end{equation}
\end{lemma}

\begin{proof}
For any ball $B\subset\cx$ and $r\in(0,\,1)$,
set
$$
h_{B,\,r}:=m_B([\cm(f\chi_{\cx\setminus 2B})]^r).
$$
Observe that, for any ball $B\subset\cx$,
\begin{align*}
&\frac{1}{\mu(5B)}\int_B|[\cm(f)(x)]^r-m_{\wz
B}([\cm(f)]^r)|\,d\mu(x)\\
&\hs\le\frac{1}{\mu(5B)}\int_B|[\cm(f)(x)]^r-h_{B,\,r}|\,d\mu(x)
+|h_{B,\,r}-h_{\wz B,\,r}|\\
&\hs\hs+\frac{1}{\mu(\wz B)}\int_{\wz B}|[\cm(f)(x)]^r-h_{\wz
B,\,r}|\,d\mu(x)
\end{align*}
and, for two doubling balls $B\subset S$,
\begin{align*}
&|m_B([\cm(f)]^r)-m_S([\cm(f)]^r)|\\
&\hs\le|m_B([\cm(f)]^r)-h_{B,\,r}|
+|h_{B,\,r}-h_{S,\,r}|+|h_{S,\,r}-m_S([\cm(f)]^r)|.
\end{align*}
Therefore, to show \eqref{e3.7}, it suffices to prove that, for all balls $B\subset\cx$,
\begin{equation}\label{e3.8}
{\rm
D_1}:=\frac{1}{\mu(5B)}\int_B|[\cm(f)(x)]^r-h_{B,\,r}|\,d\mu(x)\ls\|f\|^r_{L^\fz(\mu)}
\end{equation}
and, for all balls $B\subset S\subset\cx$ with $S$ being
$(6,\,\bz_6)$-doubling ball,
\begin{equation}\label{e3.9}
{\rm
D_2}:=|h_{B,\,r}-h_{S,\,r}|\ls[1+\dz(B,\,S)]^r\|f\|^r_{L^\fz(\mu)}.
\end{equation}

To prove \eqref{e3.8}, from the trivial inequality,
$||a|^r-|b|^r|\le|a-b|^r$ for all $a,\,b\in\cc$ and $r\in(0,\,1)$, and
the fact that $\cm$ is sublinear, we deduce that
\begin{align*}
{\rm D_1}&\le\frac{1}{\mu(5B)}
\int_B|[\cm(f)(x)]^r-[\cm(f\chi_{\cx\setminus2B})(x)]^r|\,d\mu(x)\\
&\hs+\frac{1}{\mu(5B)}\int_B|[\cm(f\chi_{\cx\setminus2B})(x)]^r-h_{B,\,r}|\,d\mu(x)\\
&\le\frac{1}{\mu(5B)}\int_B[\cm(f\chi_{2B})(x)]^r\,d\mu(x)\\
&\hs+\frac{1}{\mu(B)}\frac{1}{\mu(5B)} \int_B\int_B|\cm(f\chi_{\cx\setminus2B})(x)
-\cm(f\chi_{\cx\setminus2B})(y)|^r\,d\mu(y)\,d\mu(x)\\
&=:{\rm D_{1,\,1}}+{\rm D_{1,\,2}}.
\end{align*}

For the term ${\rm D_{1,\,1}}$, we consider the following two cases.

{\it Case} (i) $\cm$ is bounded from $H^1(\mu)$ into $L^1(\mu)$. By
Lemma \ref{l3.3}, we have
\begin{equation*}
{\rm
D_{1,\,1}}\le\frac{1}{\mu(5B)}\int_{2B}[\cm(f\chi_{2B})(x)]^r\,d\mu(x)
\ls\|f\chi_{2B}\|^r_{L^\fz(\mu)}\le\|f\|^r_{L^\fz(\mu)}.
\end{equation*}

{\it Case} (ii) $\cm$ is bounded from $L^1(\mu)$ into
$L^{1,\,\fz}(\mu)$. By the Kolmogorov inequality (see \cite[p.\,102]{d01}), we conclude that
\begin{equation*}
{\rm
D_{1,\,1}}\le\frac{[\mu(2B)]^{1-r}}{\mu(5B)}\|f\chi_{2B}\|^r_{L^1(\mu)}
\ls\|f\|^r_{L^\fz(\mu)}.
\end{equation*}
Therefore, ${\rm D_{1,\,1}}\ls\|f\|^r_{L^\fz(\mu)}$.

For the term ${\rm D_{1,\,2}}$,  by an argument used in the
estimate for \eqref{e2.13}, we see that, for all $x,\,y\in B$,
$|\cm(f\chi_{\cx\setminus2B})(x)
-\cm(f\chi_{\cx\setminus2B})(y)|\ls\|f\|_{L^\fz(\mu)}$, which
implies that ${\rm D_{1,\,2}}\ls\|f\|^r_{L^\fz(\mu)}$.

Combining the estimates for ${\rm D_{1,\,1}}$ and ${\rm D_{1,\,2}}$,
we obtain the desired estimate \eqref{e3.8}.

Now we prove \eqref{e3.9}. Write
\begin{align*}
|h_{B,\,r}-h_{S,\,r}|
&=|m_B([\cm(f\chi_{\cx\setminus2B})]^r)-m_S([\cm(f\chi_{\cx\setminus2S})]^r)|\\
&\le|m_B([\cm(f\chi_{4S\setminus2B})]^r)|+|m_S([\cm(f\chi_{4S\setminus2S})]^r)|\\
&\hs+|m_B([\cm(f\chi_{\cx\setminus4S})]^r)-m_S([\cm(f\chi_{\cx\setminus4S})]^r)|\\
&=:{\rm D_{2,\,1}}+{\rm D_{2,\,2}}+{\rm D_{2,\,3}}.
\end{align*}

Similar to the estimate for \eqref{e2.14}, we see that, for all
$x\in B$,
$\cm(f\chi_{4S\setminus2B})(x)\ls[1+\dz(B,\,S)]\|f\|_{L^\fz(\mu)}$,
which further implies that ${\rm
D_{2,\,1}}\ls[1+\dz(B,\,S)]^r\|f\|^r_{L^\fz(\mu)}$.

To estimate ${\rm D_{2,\,2}}$, notice that $S$ is a
$(6,\,\bz_6)$-doubling ball. Then, similar to the estimate for ${\rm
D_{1,\,1}}$, we have
\begin{equation*}
{\rm D_{2,\,2}}
\ls\frac{1}{\mu(6S)}\int_{4S}[\cm(f\chi_{4S\setminus2S})(x)]^r\,d\mu(x)
\ls\|f\|^r_{L^\fz(\mu)}.
\end{equation*}

Similar to the estimate for ${\rm D_{1,\,2}}$, it is easy to see
that ${\rm D_{2,\,3}}\ls\|f\|^r_{L^\fz(\mu)}$, which, together with
the estimates for ${\rm D_{2,\,1}}$ and ${\rm D_{2,\,2}}$, implies
\eqref{e3.9} and hence completes the proof of Lemma \ref{l3.4}.
\end{proof}

\begin{proof}[Proof of Theorem \ref{t1.1}]
By Theorem \ref{t2.1}, we have already known that (i) implies (ii),
(iii) and (iv). Obviously, (iii) implies (i). Therefore, to prove
Theorem \ref{t1.1}, it
suffices to prove that (ii) implies (iii) and
(iv) implies (iii).

To prove (ii) implies (iii), by the Marcinkiewicz
interpolation theorem, we only need to prove that, for all $f\in
L^p(\mu)$ with $p\in(1,\,\fz)$ and $\ell\in(0,\,\fz)$,
\begin{equation}\label{e3.10}
\mu(\{x\in\cx:\ \cm(f)(x)>\ell\})\ls\ell^{-p}\|f\|^p_{L^p(\mu)}.
\end{equation}
Let $r\in(0,\,1)$ and $N_r$ be as in \eqref{e3.2}. Notice that
$\cm(f)\le N_r(\cm(f))$ $\mu$-almost everywhere on $\cx$ and
$L^\fz_b(\mu)$ is dense in $L^p(\mu)$ for all $p\in(1,\,\fz)$. Then,
by a standard density argument, to prove \eqref{e3.10}, it suffices
to prove that, for all $f\in L_b^\fz(\mu)$ and $p\in(1,\,\fz)$,
\begin{equation}\label{e3.11}
\sup_{\ell\in(0,\,\fz)}\ell^p\mu\lf(\lf\{x\in\cx:\ N_r(\cm(f))(x)>\ell\r\}\r)
\ls\|f\|^p_{L^p(\mu)}.
\end{equation}
To this end, we consider the following two cases for $\mu(\cx)$.

{\it Case} (i) $\mu(\cx)=\fz$. Fix $\ell\in(0,\,\fz)$. For any $f\in
L^\fz_b(\mu)$, we split $f$ into $f_1$ and $f_2$ with
$f_1:=f\chi_{\{y\in\cx:\ |f(y)|>\ell\}}$ and
$f_2:=f\chi_{\{y\in\cx:\ |f(y)|\le\ell\}}$. It is easy to see that
\begin{equation}\label{e3.12}
\|f_1\|_{L^1(\mu)}\le\ell^{1-p}\|f\|^p_{L^p(\mu)},
\,\,\|f_2\|_{L^1(\mu)}\le\|f\|_{L^1(\mu)}\,\,{\rm
and}\,\,\|f_2\|_{L^\fz(\mu)}\le\ell.
\end{equation}
For each $r\in(0,\,1)$, let $M^\sharp_r$ be as in \eqref{e3.1}. From
Lemma \ref{l3.4} and \eqref{e3.12}, it follows that
$$
\|M^\sharp_r(\cm(f_2))\|_{L^\fz(\mu)}\ls\|f_2\|_{L^\fz(\mu)}\ls\ell.
$$
Hence, if $c_0$ is a sufficiently large constant, we have
\begin{equation}\label{e3.13}
\mu(\{x\in\cx:\ M^\sharp_r(\cm(f_2))(x)>c_0\ell\})=0.
\end{equation}
On the other hand, by \eqref{e3.12}, together with the boundedness from
$L^1(\mu)$ into $L^{1,\,\fz}(\mu)$ of $\cm$ and Lemma \ref{l3.2},
we see that, for any $p\in(1,\,\fz)$ and $R\in(0,\fz)$,
\begin{align*}
&\sup_{\ell\in(0,\,R)}\ell^p\mu\lf(\lf\{x\in\cx:\ N_r(\cm(f_2))(x)>\ell\r\}\r)\\
&\hs\ls\sup_{\ell\in(0,\,R)}\ell^{p-1}\sup_{\tau\in[\ell,\,\fz)}\tau\mu(\{x\in\cx:\ \cm(f_2)(x)>\tau\})
<\fz.
\end{align*}
It then follows, from the fact that $N_r\circ \cm$ is quasi-linear,
Lemma \ref{l3.1} and \eqref{e3.13}, that there exists a positive
constant $C$ such that
\begin{align}\label{e3.14}
&\sup_{\ell\in(0,\,\fz)}\ell^p\mu\lf(\lf\{x\in\cx:\ N_r(\cm(f))(x)>Cc_0\ell\r\}\r)\nonumber\\
&\hs\le\sup_{\ell\in(0,\,\fz)}\ell^p\mu\lf(\lf\{x\in\cx:\ N_r(\cm(f_2))(x)>c_0\ell\r\}\r)\nonumber\\
&\hs\hs+\sup_{\ell\in(0,\,\fz)}\ell^p\mu\lf(\lf\{x\in\cx:\ N_r(\cm(f_1))(x)>c_0\ell\r\}\r)\nonumber\\
&\hs\ls\sup_{\ell\in(0,\,\fz)}\ell^p\mu\lf(\lf\{x\in\cx:\ M^\sharp_r(\cm(f_2))(x)>c_0\ell\r\}\r)\nonumber\\
&\hs\hs+\sup_{\ell\in(0,\,\fz)}\ell^p\mu\lf(\lf\{x\in\cx:\ N_r(\cm(f_1))(x)>c_0\ell\r\}\r)\nonumber\\
&\hs\sim\sup_{\ell\in(0,\,\fz)}\ell^p\mu\lf(\lf\{x\in\cx:\ N_r(\cm(f_1))(x)>\ell\r\}\r).
\end{align}
By the boundedness from $L^1(\mu)$ into $L^{1,\,\fz}(\mu)$ of $N$,
the boundedness from $L^1(\mu)$ into $L^{1,\,\fz}(\mu)$ of $\cm$ and
\eqref{e3.12}, we conclude that
\begin{align}\label{e3.15}
&\mu(\{x\in\cx:\ N_r(\cm(f_1))(x)>\ell\})\nonumber\\
&\hs\le\mu\lf(\lf\{x\in\cx:\
N\lf([\cm(f_1)]^r
\chi_{\{y\in\cx:\ (\cm(f_1))(y)>\ell/2^{\frac1r}\}}\r)(x)>\frac{\ell^r}{2}\r\}\r)\nonumber\\
&\hs\ls\ell^{-r}\int_\cx \lf[\cm(f_1)(x)
\chi_{\{y\in\cx:\ \cm(f_1)(y)>\ell/2^{\frac1r}\}}(x)\r]^r\,d\mu(x)\nonumber\\
&\hs\ls\ell^{-r}\mu\lf(\lf\{x\in\cx:\ \cm(f_1)(x)>\ell/2^{\frac1r}\r\}\r)
\int_0^{\ell/2^{\frac1r}}s^{r-1}\,ds\nonumber\\
&\hs\hs+\ell^{-r}\int_{\ell/2^{\frac1r}}^\fz
s^{r-1}\mu(\{x\in\cx:\ \cm(f_1)(x)>s\})\,ds\nonumber\\
&\hs\ls\mu\lf(\lf\{x\in\cx:\ \cm(f_1)(x)>\ell/2^{\frac1r}\r\}\r)
+\dfrac{1}{\ell}\sup_{s\ge\ell/2^{\frac1r}}s\mu(\{x\in\cx:\ \cm(f_1)(x)>s\})\nonumber\\
&\hs\ls\dfrac{\|f_1\|_{L^1(\mu)}}{\ell}\ls\ell^{-p}\|f\|^p_{L^p(\mu)},
\end{align}
which, together with \eqref{e3.14}, implies \eqref{e3.11}.

{\it Case} (ii) $\mu(\cx)<\fz$. In this case, we assume that $f\in
L^\fz_b(\mu)$. For each fixed $\ell\in(0,\,\fz)$, with the same notation $f_1$ and
$f_2$ as in Case (i), we have $f:=f_1+f_2$. Let $r\in(0,\,1)$. We
claim that
\begin{equation}\label{e3.16}
{\rm F}:=\frac{1}{\mu(\cx)}\dint_\cx[\cm(f_2)(x)]^r\,d\mu(x)\ls
\ell^r.
\end{equation}
Indeed, by the boundedness from $L^1(\mu)$ into $L^{1,\,\fz}(\mu)$
of $\cm$, we have
\begin{align*}
&\dint_\cx[\cm(f_2)(x)]^r\,d\mu(x)\\
&\hs= r\dint_0^{\|f_2\|_{L^1(\mu)}/{\mu(\cx)}}t^{r-1}\mu
\lf(\{x\in\cx:\  \cm(f_2)(x)>t\}\r)\,dt
+r\dint_{\|f_2\|_{L^1(\mu)}/{\mu(\cx)}}^\fz \cdots\\
&\hs\ls\mu(\cx)\dint_0^{\|f_2\|_{L^1(\mu)}/{\mu(\cx)}}t^{r-1}\,dt
+\|f_2\|_{L^1(\mu)}\dint_{\|f_2\|_{L^1(\mu)}/{\mu(\cx)}}^\fz
t^{r-2}\,dt\\
&\hs\ls [\mu(\cx)]^{1-r}\|f_2\|_{L^1(\mu)}^r\ls \mu(\cx)\ell^r,
\end{align*}
which implies \eqref{e3.16}.

Observe that $\int_\cx([\cm(f_2)(x)]^r-{\rm F})\,d\mu(x)=0$ and, for any
$R\in(0,\,\fz)$,
$$
\sup_{\ell\in(0,\,R)}\ell^p\mu\lf(\lf\{x\in\cx:\ N([\cm(f_2)]^r-{\rm
F})(x)>\ell\r\}\r) \le R^p\mu(\cx)<\fz.
$$
It then follows, from Lemma \ref{l3.1}, $M^\sharp_r(\rm F)=0$,
\eqref{e3.13} and \eqref{e3.15}, that there exists a positive
constant $\wz c$ such that
\begin{align*}
&\sup_{\ell\in(0,\,\fz)}\ell^p\mu\lf(\lf\{x\in\cx:\ N_r(\cm(f))(x)>\wz
cc_0\ell\r\}\r)\\
&\hs\le\sup_{\ell\in(0,\,\fz)}\ell^p\mu\lf(\lf\{x\in\cx:\ N([\cm(f_2)]^r-{\rm
F})(x)>(c_0\ell)^r\r\}\r)\\
&\hs\hs+\sup_{\ell\in(0,\,\fz)}\ell^p\mu\lf(\lf\{x\in\cx:\ N_r(\cm(f_1))(x)>c_0\ell\r\}\r)\nonumber\\
&\hs\ls\sup_{\ell\in(0,\,\fz)}\ell^p\mu\lf(\lf\{x\in\cx:\ M^\sharp_r(\cm(f_2))(x)>c_0\ell\r\}\r)\\
&\hs\hs+\sup_{\ell\in(0,\,\fz)}\ell^p\mu\lf(\lf\{x\in\cx:\ N_r(\cm(f_1))(x)>c_0\ell\r\}\r)\nonumber\\
&\hs\sim\sup_{\ell\in(0,\,\fz)}\ell^p\mu\lf(\lf\{x\in\cx:\ N_r(\cm(f_1))(x)>\ell\r\}\r)
\ls\|f\|^p_{L^p(\mu)},
\end{align*}
where, in the first inequality, we chose $c_0$ large enough such that
${\rm F}\le (c_0\ell)^r$. This finishes the proof that
(ii) implies (iii).

Now we prove that (iv) implies (iii). To this end, we consider the following two
cases for $\mu(\cx)$.

{\it Case} (I) $\mu(\cx)=\fz$. In this case, let
$L^\fz_{b,\,0}(\mu):=\lf\{f\in L_b^\fz(\mu):\ \int_\cx
f(x)\,d\mu(x)=0\r\}$. Then, $L^\fz_{b,\,0}(\mu)$ is dense in
$L^p(\mu)$ for all $p\in(1, \fz)$. Therefore, it suffices to prove
that \eqref{e3.11} holds true for all $f\in L_{b,\,0}^\fz(\mu)$ and
$p\in(1,\,\fz)$.

For each fixed $\ell\in(0,\,\fz)$, applying Lemma \ref{l2.2}, we
conclude that $f=g+h$, where $h$ is as in Lemma \ref{l2.2} and $g:=f-h$,
such that
\begin{equation}\label{e3.17}
\|g\|_{L^\fz(\mu)}\ls\ell,
\end{equation} and
\begin{equation}\label{e3.18}
h\in H^1(\mu),\quad \|h\|_{H^1(\mu)}\ls\ell^{1-p}\|f\|^p_{L^p(\mu)}.
\end{equation}
For each $r\in(0,\,1)$, let $M^\sharp_r$ be as in \eqref{e3.1}.
Similar to \eqref{e3.13}, if $\wz c_0$ is a sufficiently large
constant, we then have
\begin{equation}\label{e3.19}
\mu(\{x\in\cx:\ M^\sharp_r(\cm(g))(x)>\wz c_0\ell\})=0.
\end{equation}
On the other hand, since both $f$ and $h$ belong to $H^1(\mu)$, we
see that $g\in H^1(\mu)$ and
$$
\|g\|_{H^1(\mu)}\le \|f\|_{H^1(\mu)}+\|h\|_{H^1(\mu)} \ls
\|f\|_{H^1(\mu)}+\ell^{1-p}\|f\|_{L^p(\mu)}^p,
$$
which, together with the boundedness from $H^1(\mu)$ into $L^1(\mu)$
of $\cm$ and Lemma \ref{l3.2}, implies that, for any $p\in(1,\,\fz)$
and $R\in(0,\,\fz)$,
\begin{align*}
&\sup_{\ell\in(0,\,R)}\ell^p\mu\lf(\lf\{x\in\cx:\ N_r(\cm(g))(x)>\ell\r\}\r)\\
&\hs\ls\sup_{\ell\in(0,\,R)}\ell^{p-1}
\sup_{\tau\in[\ell,\,\fz)}\tau\mu(\{x\in\cx:\ \cm(g)(x)>\tau\})<\fz.
\end{align*}
By some estimates similar to those of \eqref{e3.14} and \eqref{e3.15},
via the boundedness of $\cm$ from $H^1(\mu)$ into $L^1(\mu)$ and
\eqref{e3.18}, we conclude that there exists a positive constant $C$ such
that
\begin{align*}
\sup_{\ell\in(0,\,\fz)}\ell^p\mu\lf(\lf\{x\in\cx:\ N_r(\cm(f))(x)> C\wz
c_0\ell\r\}\r)
&\ls\sup_{\ell\in(0,\,\fz)}\ell^p\mu\lf(\lf\{x\in\cx:\ N_r(\cm(h))(x)>\ell\r\}\r)\\
&\ls\sup_{\ell\in(0,\,\fz)}\ell^p\dfrac{\|h\|_{H^1(\mu)}}{\ell}
\ls\|f\|^p_{L^p(\mu)},
\end{align*}
which implies that \eqref{e3.11} holds true for all $f\in
L_{b,\,0}^\fz(\mu)$ and $p\in(1,\,\fz)$.

{\it Case } (II) $\mu(\cx)<\fz$. In this case, we assume that $f\in
L^\fz_b(\mu).$ Notice that, if $\ell\in(0, \ell_0]$, where $\ell_0$
is as in Lemma \ref{l2.2}, then \eqref{e3.10} holds true trivially. Thus,
we only need to consider the case when $\ell\in(\ell_0,\,\fz)$.
For each fixed $\ell\in(\ell_0,\,\fz)$, applying Lemma \ref{l2.2}, we see that
$f=g+h$ with $g$ and $h$ satisfying \eqref{e3.17} and \eqref{e3.18}, respectively.
Let $r\in(0,\,1)$. We claim that
\begin{equation}\label{e3.20}
{\rm G}:=\frac{1}{\mu(\cx)}\dint_\cx[\cm(g)(x)]^r\,d\mu(x)\ls\ell^r.
\end{equation}
Indeed, since $\mu(\cx)<\fz$, we regard $\cx$ as a ball and the
constant function having value $[\mu(\cx)]^{-1}$ as a
$(p,\,1)_\lz$-atomic block, respectively. By \eqref{e3.17}, we have
$$
g_0:=g-\frac1{\mu(\cx)}\int_\cx g(x)\,d\mu(x)\in H^1(\mu)\,\,{\rm
and}\,\,\|g_0\|_{H^1(\mu)}\ls\ell.
$$
It then follows, from the H\"older inequality, the boundedness from
$H^1(\mu)$ into $L^1(\mu)$ of $\cm$ and \eqref{e3.17}, that
\begin{align*}
&\dint_\cx[\cm(g)(x)]^r\,d\mu(x)\\
&\hs\le[\mu(\cx)]^{1-r}\lf[\dint_\cx\cm(g)(x)\,d\mu(x)\r]^r\\
&\hs\le[\mu(\cx)]^{1-r}\lf\{\dint_\cx\cm(g_0)(x)\,d\mu(x)
+\ell\mu(\cx)\dint_\cx\cm([\mu(\cx)]^{-1})(x)\,d\mu(x)\r\}^r\\
&\hs\ls\lf\{\|g_0\|_{H^1(\mu)}+\ell\|[\mu(\cx)]^{-1}\|_{H^1(\mu)}\r\}^r\ls\ell^r,
\end{align*}
which implies \eqref{e3.20}.

Notice that $\int_\cx([\cm(g)(x)]^r-{\rm G})\,d\mu(x)=0$ and, for any
$R\in(0,\,\fz)$,
$$\sup_{\ell\in(0,\,R)}\ell^p\mu\lf(\lf\{x\in\cx:\ N([\cm(g)]^r-{\rm G})(x)>\ell\r\}\r)
\le R^p\mu(\cx)<\fz.$$ Therefore, from an argument similar to that used
in Case (ii), together with Lemma \ref{l3.1}, $M^\sharp_r(\rm G)=0$ and
\eqref{e3.19}, we deduce that
\begin{align*}
&\sup_{\ell\in(\ell_0,\,\fz)}\ell^p\mu\lf(\lf\{x\in\cx:\ N_r(\cm(f))(x)>\ell\r\}\r)\\
&\hs\ls\sup_{\ell\in(0,\,\fz)}\ell^p\mu\lf(\lf\{x\in\cx:\ N_r(\cm(h))(x)>\ell\r\}\r)
\ls\|f\|^p_{L^p(\mu)},
\end{align*}
which completes the proof that (iv) implies (iii) and hence the proof of Theorem \ref{t1.1}.
\end{proof}

\section{Proof of Theorem \ref{t1.2}\label{s4}}

To prove Theorem \ref{t1.2}, we need the following
lemma, which is a corollary of \cite[Lemma 3.2]{ly11}.

\begin{lemma}\label{l4.1}
Let $\eta\in(1,\,\fz)$ and $\bz_6$ be as in \eqref{e2.1}. Then, there exists a positive constant
$C$ such that, for all $f\in\rbmo(\mu)$ and balls $B$,
\begin{equation}\label{e4.1}
\frac{1}{\mu(\eta B)}\int_B\lf|f(y)-m_{\wz
B}(f)\r|\,d\mu(y)\le C\|f\|_{\rbmo(\mu)}
\end{equation}
and, for all $(6,\,\bz_6)$-doubling balls $B\subset S$,
\begin{equation}\label{e4.2}
|m_B(f)-m_S(f)|\le C [1+\dz(B,\,S)]\|f\|_{\rbmo(\mu)}.
\end{equation}
\end{lemma}

\begin{proof}[Proof of Theorem \ref{t1.2}]
From Theorems \ref{t1.1} and \ref{t2.1}, we deduce Theorem \ref{t1.2}(i) immediately.
To prove Theorem \ref{t1.2}(ii), we first claim that, for all $f\in L^\fz_b(\mu)$ with $\supp f\subset
B$,
\begin{equation}\label{e4.3}
\int_B\cm(f)(x)\,d\mu(x)\ls\mu(2B)\|f\|_{L^\fz(\mu)}.
\end{equation}

We consider the following two cases for $r_B$.

{\it Case} (i) $r_B\le\diam(\supp\mu)/40$. In this case, choose
$\eta=2$ in Lemma \ref{l4.1}. It then follows, from Lemma \ref{l4.1} and
\eqref{e1.8}, that
\begin{equation*}
\int_B\lf|\cm(f)(x)-m_{\wz
B}(\cm(f))\r|\,d\mu(x)\ls\mu(2B)\|\cm(f)\|_{\rbmo(\mu)}\ls\mu(2B)\|f\|_{L^\fz(\mu)}.
\end{equation*}
Therefore, the proof of \eqref{e4.3} is reduced to showing
\begin{equation}\label{e4.4}
\lf|m_{\wz B}(\cm(f))\r|\ls\|f\|_{L^\fz(\mu)}.
\end{equation}
Let $S$, $B_0$ be the same notation as in the proof of Lemma
\ref{l3.3}. Recall that $\dz(B,\,2S)\ls 1$, $\dz(B_0,\,2S)\ls 1$,
$\dz(B,\,\wz{2S})\ls 1$ and $\dz(B_0,\,\wz{2S})\ls 1$. By this, together
with Lemmas \ref{l2.1} and \ref{l4.1}, we see that
\begin{align*}
\lf|m_{B_0}(\cm(f))-m_{\wz B}(\cm(f))\r|
&\le\lf|m_{B_0}(\cm(f)-m_{\wz{2S}}(\cm(f))\r|
+\lf|m_{\wz{2S}}(\cm(f))-m_{\wz B}(\cm(f))\r|\\
&\le\lf[2+\dz(B_0,\,\wz{2S})+\dz(\wz
B,\,\wz{2S})\r]\|\cm(f)\|_{\rbmo(\mu)}\ls\|f\|_{L^\fz(\mu)},
\end{align*}
which further implies that, to prove \eqref{e4.4}, it suffices to prove that
\begin{equation}\label{e4.5}
\lf|m_{B_0}(\cm(f))\r|\ls\|f\|_{L^\fz(\mu)}.
\end{equation}
Notice that, for all $y\in B_0$ and $z\in B$, it holds true that $d(y,\,z)\ge r_B/2$
and hence $d(c_B,\,y)\le d(c_B,\,z)+d(z,\,y)\ls d(y,\,z)$. By the
Minkowski inequality, \eqref{e1.5}, \eqref{e1.3}, \eqref{e1.4} and
the fact that $\supp f\subset B$, we conclude that, for all $y\in B_0$,
\begin{align*}
\cm(f)(y)&=\lf[\int_0^\fz\lf|\int_{d(y,\,z)<
t}K(y,\,z)f(z)\,d\mu(z)\r|^2\,\frac{dt}{t^3}\r]^{1/2}\\
&\ls\|f\|_{L^\fz(\mu)}\int_{B}\lf[\int_{d(y,\,z)}^\fz\frac{dt}{t^3}\r]^{1/2}
\frac{d(y,\,z)}{\lz(y,\,d(y,\,z))}d\mu(z)\\
&\ls\|f\|_{L^\fz(\mu)}\int_{B}\frac{1}{\lz(c_B,\,d(y,\,z))}d\mu(z)
\ls\|f\|_{L^\fz(\mu)}\frac{\mu(B)}{\lz(c_B,\,r_B)}
\ls\|f\|_{L^\fz(\mu)},
\end{align*}
which implies \eqref{e4.5}. Hence, \eqref{e4.3} holds true in this case.

{\it Case} (ii) $r_B>\diam(\supp\mu)/40$. In this case, the argument
is almost the same as the one of Case (ii) in the proof of Lemma
\ref{l3.3}. We omit the details, which shows that the claim
\eqref{e4.3} also holds true in this case.

Now based on the claim \eqref{e4.3}, we prove Theorem \ref{t1.2}(ii).
Take $\rho=4$ and $p=\fz$ in Definition \ref{d1.4}. By the
definition of $H^{1,\,\fz}_{\rm fin}(\mu)$, it suffices to show that,
for any $(\fz,\,1)_\lz$-atomic block $b$,
\begin{equation}\label{e4.6}
\|\cm(b)\|_{L^1(\mu)}\ls|b|_{H^{1,\,\fz}_{\rm atb}(\mu)}.
\end{equation}
By the argument used in the estimate for \eqref{e2.7}, we see that
\eqref{e4.6} holds true if we replace any $(p_0,\,1)_\lz$-atomic block and
\eqref{e2.9} by an $(\fz,\,1)_\lz$-atomic block and \eqref{e4.3},
respectively. We omit the details, which completes the proof of
Theorem \ref{t1.2}.
\end{proof}

\Acknowledgements{The first author is supported by the Mathematical Tianyuan Youth
Fund of the National Natural Science Foundation of
China (Grant No. 11026120) and Chinese Universities Scientific Fund (Grant No.
2011JS043). The second author is supported by the National Natural
Science Foundation of China (Grant No. 11171027) and
the Specialized Research Fund for the Doctoral Program of Higher Education
of China (Grant No. 20120003110003). The authors would like to thank the referees for their careful reading
and many valuable remarks which made this article more readable.}


\end{document}